\documentclass[10pt,a4paper,reqno]{article}
\usepackage{graphicx} 
\usepackage[english]{babel} 
\usepackage{csquotes} 
\usepackage{filecontents} 
\usepackage[dvipsnames]{xcolor}
\usepackage{color}
\usepackage[colorinlistoftodos]{todonotes}

\usepackage{microtype}

\usepackage[
  backend=biber,
  hyperref=true,
  backref=true,
  isbn=false,
  doi=true,
  url=false,
  natbib=true,
  eprint=true,
  useprefix=true,
  maxcitenames=99,
  maxbibnames=99,  
  maxalphanames=99, 
  minalphanames=99,
  safeinputenc,
  style=alphabetic,
  citestyle=alphabetic,
  block=space,
  datamodel=ext-eprint,
]{biblatex}
\usepackage[
  hypertexnames = false,
  colorlinks    = true,
  citecolor     = blue,
  linkcolor     = blue,
  urlcolor      = blue,
  linktocpage = true,
  breaklinks
  ]{hyperref}
\renewbibmacro{in:}{} 

\DeclareSourcemap{
  \maps[datatype=bibtex]{
    \map{
      \step[fieldsource=pmid, fieldtarget=pubmed]
    }
  }
}

\makeatletter
\DeclareFieldFormat{arxiv}{%
  arXiv\addcolon\space
  \ifhyperref
    {\href{http://arxiv.org/\abx@arxivpath/#1}{%
       \nolinkurl{#1}%
       \iffieldundef{arxivclass}
         {}
         {\addspace\texttt{\mkbibbrackets{\thefield{arxivclass}}}}}}
    {\nolinkurl{#1}
     \iffieldundef{arxivclass}
       {}
       {\addspace\texttt{\mkbibbrackets{\thefield{arxivclass}}}}}}
\makeatother
\DeclareFieldFormat{pmcid}{%
  PMCID\addcolon\space
  \ifhyperref
    {\href{http://www.ncbi.nlm.nih.gov/pmc/articles/#1}{\nolinkurl{#1}}}
    {\nolinkurl{#1}}}
\DeclareFieldFormat{mrnumber}{%
  MR\addcolon\space
  \ifhyperref
    {\href{http://www.ams.org/mathscinet-getitem?mr=MR#1}{\nolinkurl{#1}}}
    {\nolinkurl{#1}}}
\DeclareFieldFormat{zbl}{%
  Zbl\addcolon\space
  \ifhyperref
    {\href{http://zbmath.org/?q=an:#1}{\nolinkurl{#1}}}
    {\nolinkurl{#1}}}
\DeclareFieldAlias{jstor}{eprint:jstor}
\DeclareFieldAlias{hdl}{eprint:hdl}
\DeclareFieldAlias{pubmed}{eprint:pubmed}
\DeclareFieldAlias{googlebooks}{eprint:googlebooks}

\renewbibmacro*{eprint}{%
  \printfield{arxiv}%
  \newunit\newblock
  \printfield{jstor}%
  \newunit\newblock
  \printfield{mrnumber}%
  \newunit\newblock
  \printfield{zbl}%
  \newunit\newblock
  \printfield{hdl}%
  \newunit\newblock
  \printfield{pubmed}%
  \newunit\newblock
  \printfield{pmcid}%
  \newunit\newblock
  \printfield{googlebooks}%
  \newunit\newblock
  \iffieldundef{eprinttype}
    {\printfield{eprint}}
    {\printfield[eprint:\strfield{eprinttype}]{eprint}}}

\usepackage[margin=2.2cm,includehead]{geometry}

\usepackage{amsmath,stmaryrd, upgreek}
\usepackage{amsthm,amssymb}
\usepackage{latexsym}
\usepackage{amscd}
\usepackage{mathrsfs}
\usepackage{url}
\usepackage{mathtools}
\usepackage{cleveref}
\usepackage{doi}

\usepackage[T1]{fontenc}
\usepackage{tikz-cd}
\usepackage{titlesec}

\usepackage[titletoc,toc,title]{appendix}

\makeatletter 
\renewcommand\tableofcontents{%
    \@starttoc{toc}%

}
\makeatother

\usepackage{enumerate}
\usepackage{slashed}
\graphicspath{}
\usepackage{fancyhdr} 

\usepackage{changepage}

\newcommand\shorttitle{A gradient flow of Spin(7)-structures} 
\newcommand\authors{S. Dwivedi} 
\newcommand{\rest}[2]{ { \left. {#1} \right|}_{{#2} }}

\newcounter{commentCounter}
\titleformat{\section}    
       {\normalfont\large\bfseries\center}{\thesection.}{1em}{}
       
\makeatletter
\newcommand*{\rom}[1]{\expandafter\@slowromancap\romannumeral #1@}
\makeatother

\newtheorem{theorem}{Theorem}[section]

\newtheorem{lemma}[theorem]{Lemma}
\newtheorem{proposition}[theorem]{Proposition}

\theoremstyle{definition}
\newtheorem{definition}[theorem]{Definition}
\newtheorem{remark}[theorem]{Remark}

\numberwithin{equation}{section}
\def\bR{\mathbb R}

\newcommand{\mb}[1]{\mathbf{#1}}

\def\wtd{\widetilde}

\def\pt{\partial}
\def\del{\nabla}
\def\G2{\mathrm{G}_2}
\def\g2{\varphi}
\def\S7{\mathrm{Spin}(7)}
\def\s7{\Phi}

\def\cL{\mathcal{L}}

\def\cS{\mathcal{S}}

\def\Spin7{\mathrm{Spin(7)}}
\def\Ric{\mathrm{Ric}}
\def\Riem{\mathrm{Rm}}
\def\SO{\mathrm{SO}}

\def\dots7{\Dot{\Phi}}

\DeclareMathOperator\Div{div}

\DeclareMathOperator\vol{vol}

\DeclareMathOperator\tr{tr}

\DeclareMathOperator\Id{Id}

\usepackage{cancel}
\usepackage[all]{xy}
\usepackage{multicol}

\usepackage{charter}

\begin{document}

\title{
A gradient flow of Spin(7)-structures}
\author{Shubham Dwivedi}
\date{}

\maketitle

\begin{abstract}
\noindent
We formulate and study the negative gradient flow of an energy functional of Spin(7)-structures on compact
$8$-manifolds. The energy functional is the $L^2$-norm of the torsion of the Spin(7)-structure. Our main result is the short-time existence and uniqueness of solutions to the flow. We also explain how this negative gradient flow contains, as the highest order terms, all independent second order differential invariants of Spin(7)-structures which can be made into an admissible $4$-form. We also study solitons of the flow and prove a non-existence result for compact expanding solitons.
\end{abstract}


\let\thefootnote\relax

\section{Introduction}\label{sec:intro}

Given an $8$-dimensional smooth manifold $M$, a Spin(7)-structure on $M$ is a reduction of the structure group of the frame bundle $\text{Fr}(M)$ to the Lie group Spin(7) $\subset \mathrm{SO}(8)$. While many $8$-manifolds satisfying the topological condition \eqref{eq:topcond} admit Spin(7)-structures, the existence of \emph{torsion-free} Spin(7)-structures is a difficult problem. Given the success of geometric flows in proving the existence of special geometric structures on manifolds it is natural to attempt to tackle such questions using a suitable geometric flow of Spin(7)-structures.

\medskip

\noindent
Unlike the case of geometric flows of $\G2$-structures, there have been fewer studies of flows of Spin(7)-structures. A flow of isometric/harmonic Spin(7)-structures was introduced and studied in detail in \cite{dle-isometric}. The authors proved various analytic and geometric results for the flow. This flow can be seen as an instance of the more general theory of \emph{harmonic} $H$-structures which has been studied in a fairly detailed manner in \cite{loubeau-saearp} and \cite{FLMS}. The isometric flow enjoys many nice analytical properties and has long time existence and convergence to a critical point under mild initial conditions (see the above papers for details). However, one disadvantage of the isometric flow is that it runs in an \emph{isometry} class of a given Spin(7)-structure and hence the metric remains fixed throughout the flow. We remedy this here by studying the most general flow of Spin(7)-structures (in particular, the metric is also changing) which include the isometric flow as a special case. On a fixed oriented spin Riemannian $8$-manifold, a Spin(7)-structure is equivalent (up to sign) to a unit spinor field. Using a spinorial approach, Ammann--Weiss--Witt \cite{AWW} also studied the negative gradient flow of a Dirichlet energy, thought of as a function of a unit spinor field. They proved general short-time existence and uniqueness. In contrast to their work, our approach is more direct and uses the Riemannian geometric properties of a Spin(7)-structure with the flow equation and related quantities all given in terms of the curvature of the underlying metric and the torsion of the Spin(7)-structure. We also describe all possible second order quasilinear flows of Spin(7)-structures. A study of some natural functionals of Spin(7)-structures and the corresponding Euler-Lagrange equations has also been done recently by Krasnov \cite{Krasnov}.

\medskip

\noindent
Our point of view of studying general flows of Spin(7)-structures is that of \cite{dgk-flows} where the authors study the most general second order flows of $\G2$-structures by classifying all linearly independent second order (in the $\G2$-structure) tensors coming from the Riemann curvature tensor and the covariant derivative of the torsion which can be made into a $3$-form.

\medskip

\noindent
We now describe the energy functional and its negative gradient flow. More details are given in \textsection~\ref{sec: Basics of the Spin(7)-flow}.

\begin{definition}
Let $M^8$ be a compact manifold which admits Spin(7)-structures. The \emph{energy functional} $E$ on the set of Spin(7)-structures on $M$ is
\begin{align}
\label{E1}
    E(\s7)=\frac 12 \int_M |T_{\s7}|^2 \vol_{\s7}     
\end{align}
where $T_{\s7}$ is the torsion of the Spin(7)-structure $\s7$ and the norm and the volume form are with respect to the metric induced by $\s7$. The critical points of the functional $E$ are torsion-free Spin(7)-structures which are also absolute minimizers. We call $E$ an \enquote{energy functional} because it follows from \eqref{eq:Texpress} that upto some constant, $E$ is essentially $\int_M |\del \Phi|^2 \vol$ and thus is similar to the Dirichlet energy functional for functions.
\end{definition}

\noindent
We compute the first variation of $E$ in Lemma~\ref{lemma:Energyvar} which also gives the negative gradient of the energy functional. We describe the negative gradient flow below. Let $T$ denote the torsion of the Spin(7)-structure $\Phi$. See \textsection~\ref{sec:prelims} and \textsection~\ref{sec: Basics of the Spin(7)-flow} for the definition of the terms and more details.

\begin{definition}[Negative gradient flow]
Let $(M^8, \s7_0)$ be a compact manifold with a Spin(7)-structure. 
The negative gradient flow of the energy functional $E$ is the flow of Spin(7)-structures which is the  following initial value problem
\begin{align} 
\label{F1} 
 \left\{\begin{array}{rl} 
      & \dfrac{\pt \s7}{\pt t} = \left(-\Ric+2(\cL_{T_8}g) + T*T -|T|^2g + 2 \Div T\right) \diamond \s7, \\
      & \s7(0) =\s7_0.
    \tag{GF}
   \end{array}\right.
\end{align}
where $\Ric$ is the Ricci curvature of the underlying metric, $T_8$ is the $8$-dimensional component of the torsion tensor $T$, $T*T$ is the tensor described in Definition~\ref{def:ene} and $\diamond$ is the operator described in \eqref{eq:diadefn1}. Here $\Div T$ is the divergence of the torsion tensor and is defined in \eqref{eq:divT}.
\end{definition}

\medskip

\noindent
Given any geometric flow, one of the main questions is whether there exist a solution to the flow for a short time and if any two solutions starting with the same initial conditions are unique. The main theorem of the paper is as follows.

\medskip

\noindent
{\textbf{Theorem~\ref{thm:ste}.}} \emph{Let $(M^8, \Phi_0)$ be a compact $8$-manifold with a Spin(7)-structure $\Phi_0$ and consider the negative gradient flow of the natural energy functional $E$ in \eqref{E1}. This is the flow \eqref{gfloweqn}. Then there exists $\varepsilon >0$ and a unique smooth solution $\Phi(t)$ of \eqref{gfloweqn} for $t\in [0, \varepsilon)$ with $\varepsilon=\varepsilon(\Phi_0)$.}

\medskip

\noindent
We prove the main theorem by explicitly calculating the principal symbols of the operators involved in the flow and then using a modified DeTurck's trick.

\medskip

\noindent
As discussed at the end of \textsection~\ref{sec:prelims}, the most general flow of Spin(7)-structures which is second order can be written as
\begin{align}\label{genflow}
\frac{\pt}{\pt t}\Phi(t)=(-a\Ric+b\cL_{T_8}g + c\Div T + \text{l.o.t})\diamond \Phi    
\end{align}
where $a, b, c\in \bR$ and $\text{l.o.t}$ are the terms which are lower order in $\Phi$. The situation in the Spin(7)-case is a bit simpler than the corresponding case of the $\G2$-structures where there are more terms which could appear in a flow of $\G2$-structures \cite[\textsection 1]{dgk-flows}. This is because of the fact that for $p\in M$, the subbundle $A_pM$ of \emph{admissible} 4-forms in $\Lambda^4T^*_pM$ has codimension $27$ and thus any variation of a Spin(7)-structure is given only by $4$-forms in $\Omega^4_{1+7+35}$. Consequently, many naturally occurring tensors, for example the $27$-dimensional Weyl curvature tensor, cannot be a candidate for a flow of Spin(7)-structures (see \textsection~\ref{sec:prelims} for details) even though they can be made into a $4$-form. 

\medskip

\noindent
The methods in \textsection~\ref{sec:ste} can be used to establish short-time existence and uniqueness results for the general flow in \eqref{genflow} subjected to conditions on the coefficients $a, b$ and $c$ which will be sufficient conditions to apply the modified DeTurck's trick described in \textsection~\ref{subsec:mdtt}. The advantage of looking at a specific flow in the family \eqref{genflow}, which for us is \eqref{gfloweqn} which corresponds to $a=-1, b, c=2$ is that it also provides us with the lower order terms which are important when one wants to study further analytic properties of the flow. An example of the importance of lower order terms is our result on the non-existence of compact expanding solitons for the flow in \textsection~\ref{sec:solitons}.

\medskip

\noindent
Solitons of a geometric flow are special solutions which move only by scalings and diffeomorphisms. These are also called self-similar solutions. A soliton can be viewed as a generalized fixed point of the flow modulo scalings and diffeomorphisms. These special solutions are important because they serve as singularity models for the flow. We study solitons of the flow \eqref{gfloweqn}. We derive equations for the underlying metric and the divergence of the torsion of a soliton solution. We expect that if we have a Cheeger--Gromov--Hamilton type compactness theorem for the solutions of the flow then the solitons will play a role in understanding the behaviour of the flow near a singularity.  

\medskip

\noindent
The paper is organized as follows. We start by reviewing the notion of Spin(7)-structures, the decomposition of space of forms and the torsion tensor in \textsection~\ref{sec:prelims}. The first variation of the energy functional is computed in \textsection~\ref{sec: Basics of the Spin(7)-flow}. We give a brief review of elliptic and parabolic differential operators in \textsection~\ref{sec:diff-ops} which is followed by the computations of principal symbols of first and second order nonlinear operators in \textsection~\ref{subsec:psymb}. In \textsection~\ref{subsec:ste}, we show that the failure of parabolicity of the negative gradient flow of the energy functional is only due to the diffeomorphism invariance of the tensors involved and this leads to a modified DeTurck's trick in \textsection~\ref{subsec:mdtt}. The main theorem, Theorem~\ref{thm:ste}, is proved in \textsection~\ref{subsec:stef}. Finally, in \textsection~\ref{sec:solitons} we study solitons of the flow and prove Proposition~\ref{prop:solmainprop} which is a non-existence result for compact expanding solitons of the flow.

\medskip

\noindent
Throughout the paper, we compute in a local orthonormal frame, so all indices are subscripts and any
repeated indices are summed over all values from $1$ to $8$. Our convention for labelling the Riemann curvature tensor is
\begin{align*}
R_{ijkl}=g\left(\del_{e_i}(\del_{e_j} e_k)-\del_{e_j}(\del_{e_i} e_k)-\del_{[e_i, e_j]} e_k,e_l\right)  
\end{align*}
in a local orthonormal frame and as a result, the Ricci tensor is $R_{ij}=R_{lijl}$. We also have the Ricci identity which, for instance, for a $2$-tensor $S$ reads as
\begin{align}\label{eq:ricciiden}
\del_i\del_jS_{kl}-\del_j\del_iS_{kl}=-R_{ijkm}S_{ml}-R_{ijlm}S_{km}. 
\end{align}

\medskip

\noindent
\textbf{Acknowledgements.} We are grateful to Panagiotis Gianniotis and Spiro Karigiannis for various discussions on this project in particular and on flows of special geometric structures in general. We would like to thank Thomas Walpuski for discussions on the torsion decomposition of Spin(7)-structures, Ragini Singhal for explaining the representation theoretic aspects of Spin(7)-structures and Udhav Fowdar for his interest in the paper and for catching typos. Finally, we are grateful to the anonymous referees for a very careful reading of the paper and for their various suggestions which have improved the paper. A part of this work was done when the author was a Simons-CRM Scholar-in-Residence at the CRM, Montreal. We would like to thank the Simons Foundation and the CRM for funding and CRM for providing a wonderful environment to work in.

\section{Preliminaries on Spin(7)-structures}\label{sec:prelims}

In this section, we briefly review the notion of a Spin(7)-structure on an $8$-dimensional manifold $M$ and the associated decomposition of differential forms. We also discuss the torsion tensor of a Spin(7)-structure. More details can be found in \cite[Chapter 10]{joycebook} and \cite{karigiannis-spin7}.

\medskip

 \noindent
 A Spin(7)-structure on $M$ is a particular type of $4$-form $\s7$ on $M$. The existence of such a structure is a \emph{topological condition}. Concretely, an $8$-manifold admits a Spin(7)-structure if and only if it is orientable, spinnable, conditions which are equivalent to the vanishing of the first and second Stiefel--Whitney classes respectively, and for some orientation on M,
 \begin{align}\label{eq:topcond}
 p_1^2-4p_2+8\chi=0,
\end{align}where $p_i$ is the i-th Pontryagin class and $\chi$ is the Euler class of $M$ \cite[Theorem 10.7]{lawson-michelsohn}. The space of $4$-forms which determine a Spin(7)-structure on $M$ is a subbundle $A$ of $\Omega^4(M)$, called the bundle of \emph{admissible} $4$-forms. This is \emph{not} a vector subbundle and it is not even an open subbundle, unlike the case for $\G2$-structures. For $p\in M$, the subbundle $A_p(M)$ is of codimension $27$ in $\Lambda^4(T^*_pM)$.

\medskip

\noindent
A Spin(7)-structure $\s7$  determines a Riemannian metric and an orientation on $M$ in a nonlinear way which can be described in local frame as follows. Let $p\in M$ and extend  a non-zero tangent vector $v\in T_pM$  to a local frame $\{v, e_1, \cdots , e_7\}$ near $p$. Define
\begin{align*}
    B_{ij}(v)&=((e_i\lrcorner v\lrcorner \s7)\wedge (e_j\lrcorner v\lrcorner \s7)\wedge (v\lrcorner \s7))(e_1, \cdots , e_7),\\
    A(v)&=((v\lrcorner \s7)\wedge \s7)(e_1, \cdots, e_7).
\end{align*}
Then the metric induced by $\s7$ is given by
\begin{align}
\label{eq:metric}
    (g_{\s7}(v,v))^2=-\frac{7^3}{6^{\frac{7}{3}}}\frac{(\textup{det}\ B_{ij}(v))^{\frac 13}}{A(v)^3}.    
\end{align}
The metric and the orientation determine a Hodge star operator $\star$, and the $4$-form is \emph{self-dual}, i.e., $\star \s7=\s7$. 

\begin{definition}
    Let $\del$ be the Levi-Civita connection of the metric $g_{\s7}$. The pair $(M^8, \s7)$ is a \emph{Spin(7)-manifold} if $\del \s7=0$. This is a non-linear partial differential equation for $\s7$, since $\del$ depends on $g$, which in turn depends non-linearly on $\s7$. A Spin(7)-manifold has Riemannian holonomy contained in the subgroup $\S7\subset \SO(8)$. Such a parallel Spin(7)-structure is also called \emph{torsion free}. 
\end{definition}

\subsection{Decomposition of the space of forms}
\label{subsec:formdecomp}

The existence of a Spin(7)-structure $\s7$ induces a decomposition of the space of differential forms on $M$  into irreducible Spin(7)-representations. We have the following orthogonal decompositions, with respect to $g_\s7$:
\begin{align*}
\Omega^2=\Omega^2_{7}\oplus \Omega^2_{21},\ \ \ \ \ \ \ \Omega^3=\Omega^3_{8}\oplus \Omega^3_{48}, \ \ \ \ \ \ \ \ \  \Omega^4=\Omega^4_{1}\oplus \Omega^4_{7}\oplus \Omega^4_{27}\oplus \Omega^4_{35},
\end{align*}
where $\Omega^k_l$ has pointwise dimension $l$. Explicitly, $\Omega^2$ and $\Omega^3$ are described as follows:
\begin{align}
    \Omega^2_7&=\{\beta \in \Omega^2 \mid \star(\s7 \wedge \beta)=-3\beta\}, \label{eq:27decomp1}\\
    \Omega^2_{21}&=\{ \beta\in \Omega^2 \mid \star(\s7\wedge \beta)=\beta\}, \label{eq:221decomp1}
\end{align}
and
\begin{align}
    \Omega^3_8&=\{ X\lrcorner \s7 \mid X\in \Gamma(TM)\}, \label{eq:38decomp1}\\
    \Omega^3_{48}&=\{ \gamma \in \Omega^3\mid \gamma \wedge \s7 =0\}. \label{eq:348decomp1}
\end{align}
For our computations, it is useful to describe the spaces of forms in local frame. For $\beta\in \Omega^2(M)$,
\begin{align}
    \beta_{ij}\in \Omega^2_7 \iff \beta_{ab}\s7_{abij}&=-6\beta_{ij},\label{eq:27decomp2}\\
    \beta_{ij}\in \Omega^2_{21} \iff \beta_{ab}\s7_{abij}&=2\beta_{ij}.\label{eq:221decomp2}
\end{align}

\begin{remark}
A convention different than ours is also prevalent in the literature. In this convention, $\Omega^2_7$ and $\Omega^2_{21}$ are the $+3$ and $-1$ eigenspaces of the map $\beta\mapsto *(\Phi\wedge \beta),$ respectively. As a result, the constants on the right hand sides of \eqref{eq:27decomp2} and \eqref{eq:221decomp2} are $+6$ and $-2$ respectively.  
\end{remark}

\noindent
For $\gamma\in \Omega^3(M)$,
\begin{align}
    \gamma_{ijk}\in \Omega^3_8 &\iff \gamma_{ijk}=X_l\s7_{ijkl}\ \ \ \ \textup{for\ some}\ X\in \Gamma(TM), \label{eq:38decomp2}\\
    \gamma_{ijk}\in \Omega^3_{48} &\iff \gamma_{ijk}\s7_{ijkl}=0. \label{eq:348decomp2}
\end{align}
If $\pi_7$ and $\pi_{21}$ are the projection operators on $\Omega^2$, it follows from \eqref{eq:27decomp2} and \eqref{eq:221decomp2} that 
\begin{align}
\pi_7(\beta)_{ij}&=\frac 14\beta_{ij}-\frac 18\beta_{ab}\s7_{abij}, \label{eq:pi7}\\
\pi_{21}(\beta)_{ij}&=\frac 34\beta_{ij}+\frac 18\beta_{ab}\s7_{abij}. \label{eq:pi21}
\end{align}
We will be using these equations throughout the paper.
Finally, for $\beta_{ij}\in \Omega^2_{21}$,
\begin{align}
    \beta_{ab}\s7_{bpqr}&=\beta_{pi}\s7_{iqra}+\beta_{qi}\s7_{irpa}+\beta_{ri}\s7_{ipqa}, \label{eq:221prop}   
\end{align}
which can be used to show that $\Omega^2_{21}\equiv \mathfrak{so}(7)$ is the Lie algebra of Spin(7), with the commutator of matrices
\begin{align*}
    [\alpha, \beta]_{ij}=\alpha_{il}\beta_{lj}-\alpha_{jl}\beta_{li}.
\end{align*}

\medskip

\noindent
To describe $\Omega^4$ in local orthonormal frame, we use the $\diamond$ operator which was first described in \cite{dgk-isometric} for the $\G2$ case and in \cite{dle-isometric} for the Spin(7) (see also \cite[eq. (1.14)]{FLMS} for the general case of $H$-structures). Given $A\in \Gamma(T^*M\otimes TM)$, define
\begin{align}
\label{eq:diadefn1}
    A\diamond \s7= \frac{1}{24}(A_{ip}\s7_{pjkl}+A_{jp}\s7_{ipkl}+A_{kp}\s7_{ijpl}+A_{lp}\s7_{ijkp})e^i\wedge e^j\wedge e^k\wedge e^l,    
\end{align}
and hence 
\begin{align}
\label{eq:diadefn2}
    (A\diamond \s7)_{ijkl}=  A_{ip}\s7_{pjkl}+A_{jp}\s7_{ipkl}+A_{kp}\s7_{ijpl}+A_{lp}\s7_{ijkp}. 
\end{align}
Recall that $
    \Gamma(T^*M\otimes TM)=\Omega^0\oplus S_0\oplus \Omega^2$, where $S_0$ is the space of trace-free symmetric $2$-tensors and $\Omega^2$ further splits orthogonally into \eqref{eq:27decomp1} and \eqref{eq:221decomp1}, so
\begin{align}
\label{eq:splitting TM* x TM}
 \Gamma(T^*M\otimes TM)=\Omega^0\oplus S_0 \oplus \Omega^2_7\oplus \Omega^2_{21}.   
\end{align}
With respect to this splitting, we can write $A=\frac 18 (\tr A)g+A_{35}+A_7+A_{21}$ where $A_{35}$ is a symmetric traceless $2$-tensor. 
The diamond contraction \eqref{eq:diadefn2} defines a linear map $A\mapsto A\diamond \s7$, from $\Omega^0\oplus S_0 \oplus \Omega^2_7\oplus \Omega^2_{21}$ to $\Omega^4(M)$. We record the following properties of the $\diamond$ operator whose proof can be found in \cite{karigiannis-spin7} or \cite{dle-isometric}.

\begin{proposition}\label{prop:diaproperties2}
Let $(M, \s7)$ be a manifold with a Spin(7)-structure. Then
\begin{enumerate}
\item The Hodge star of $A\diamond \s7$ is
\begin{align}\label{diapropertieseqn1}
\star(A\diamond \s7)&=(\Bar{A}\diamond \s7)\ \  \textup{where}\ \  \Bar{A}_{ij}=\frac 14 (\tr A)g_{ij}-A_{ji}.    
\end{align}

\item If $B\in \Gamma(T^*M\otimes TM)$ and $A_{35},\  B_{35}$ denote the traceless symmetric parts of $A$ and $B$ respectively and $A_7,\ B_7$ denote their $\Omega^2_7$ component, then
\begin{align}\label{diapropertieseqn2}
\langle A\diamond \s7, B\diamond \s7 \rangle &= 84(\tr A)(\tr B)+96\langle A_{35}, B_{35}\rangle + 384 \langle A_7, B_7 \rangle . 
\end{align}
\qed
\end{enumerate}
\end{proposition}
\noindent
More properties of the $\diamond$ operator for general $H$-structures have been proved in \cite[Lemma 1.4]{FLMS}.

\medskip

\noindent
The following proposition was originally proved  in \cite[Prop. 2.3]{karigiannis-spin7}.
\begin{proposition}\label{prop:diaproperties1}
The kernel of the map $A\mapsto A\diamond \s7$ is isomorphic to the subspace $\Omega^2_{21}$. The remaining three summands $\Omega^0,\ S_0$ and $\Omega^2_7$ are mapped isomorphically onto the subspaces $\Omega^4_1,\ \Omega^4_{35}$ and $\Omega^4_7$ respectively.
\end{proposition}

\begin{proof}
Using \eqref{diapropertieseqn2} with $A=B$ gives
\begin{align*}
|A\diamond \s7|^2&=84 (\tr A)^2 + 96|A_{35}|^2+384|A_7|^2    
\end{align*}
and hence $A\diamond \s7=0 \iff A=A_{21}$. If $A_{21}=0$ then since all the coefficients in $|A\diamond \s7|^2$ are positive hence the map $A\mapsto A\diamond \s7$ is injective on $(\Omega^2_{21})^{\perp}$ and by dimension count, the map is a linear isomorphism.
\end{proof}

\noindent
For a concrete description of the space $\Omega^4_{27}$ we refer the reader to \cite[pg. 4-5]{karigiannis-spin7}. The decomposition of $\Omega^4(M)$ into self-dual and anti-self-dual parts is
 \begin{align}\label{eq:sd&asd}
\Omega^4_+=\{\sigma \in \Omega^4\mid \star \sigma = \sigma\}=\Omega^4_1\oplus \Omega^4_7\oplus \Omega^4_{27},\ \ \ \ \ \Omega^4_-=\{\sigma\in \Omega^4\mid \star \sigma=-\sigma\}=\Omega^4_{35} .    
 \end{align}


\noindent

\medskip

\noindent
Before we discuss the torsion of a Spin(7)-structure, we note some contraction identities involving the $4$-form $\s7$. In local coordinates $\{x^1, \cdots, x^8\}$, the $4$-form $\s7$ is
\begin{align*}
\s7=\frac{1}{24}\s7_{ijkl}\ dx^i\wedge dx^j\wedge dx^k\wedge dx^l    
\end{align*}
where $\s7_{ijkl}$ is totally skew-symmetric. We have the following identities
\begin{align}
    \s7_{ijkl}\s7_{abkl}&=6g_{ia}g_{jb}-6g_{ib}g_{ja}-4\s7_{ijab}, \label{eq:impiden2} \\
    \s7_{ijkl}\s7_{ajkl}&=42g_{ia}, \label{eq:impiden3} \\
    \s7_{ijkl}\s7_{ijkl}&=336 . \label{eq:impiden4}
\end{align}
We also have contraction identities involving $\del \s7$ and $\s7$
\begin{align}
(\del_m\s7_{ijkl})\s7_{abkl}&=-\s7_{ijkl}(\del_m\s7_{abkl})-4\del_m\s7_{ijab} \label{eq:impiden5}\\
(\del_m\s7_{ijkl})\s7_{ajkl}&=-\s7_{ijkl}(\del_m\s7_{ajkl}) \label{eq:impiden6} \\
(\del_m\s7_{ijkl})\s7_{ijkl}&=0. \label{eq:impiden7}
\end{align}

\noindent
We now describe the \emph{torsion} of a Spin(7)-structure. Given $X\in \Gamma(TM)$, we know from \cite[Lemma 2.10]{karigiannis-spin7} that $\del_X\s7$ lies in the subbundle $\Omega^4_7\subset\Omega^4$. 
\begin{definition}
    The \emph{torsion tensor} of a Spin(7)-structure $\s7$ is the element of $\Omega^1_8\otimes \Omega^2_7$ defined by expressing $\del \s7$.
    Since $\del_X\s7\in \Omega^4_7$, by Proposition \ref{prop:diaproperties1}, $\del \s7$ can be written as
\begin{align}
\label{Tdefneqn}
    \del_m\s7_{ijkl}=(T_m\diamond \s7)_{ijkl}=T_{m;ip}\s7_{pjkl}+T_{m;jp}\s7_{ipkl}+T_{m;kp}\s7_{ijpl}+T_{m;lp}\s7_{ijkp}    
\end{align}
    where $T_{m;ab}\in\Omega^2_7$, for each fixed $m$. This defines the torsion tensor $T$ of a Spin(7)-structure, which is an element of $\Omega^1_8\otimes \Omega^2_7$.
\end{definition}

\noindent
In terms of $\del \s7$, the torsion $T$ is given by
\begin{align}
\label{eq:Texpress}
    T_{m;ab} 
    =\frac{1}{96}(\del_m\s7_{ajkl})\s7_{bjkl}    
\end{align}
since $T$ is an element of $\Omega^1_8\otimes \Omega^2_7$.
\begin{remark}
The notation $T_{m;ab}$ should not be confused with taking two covariant derivatives of $T_m$. The torsion tensor T is an element of $\Omega^1_8\otimes \Omega^2_7$ and thus, for each fixed index $m$, $T_{m;ab}\in \Omega^2_7$.
\end{remark} 
\noindent
We write $\Div T$ for the divergence of the torsion which is an element of $\Omega^2_7$ and is given by
\begin{align}\label{eq:divT}
(\Div T)_{jk}=\del_mT_{m;jk}.
\end{align}We also have the following result, originally due to Fern\'andez \cite{fernandez-spin7}.

\begin{theorem}\cite{fernandez-spin7}
The Spin(7)-structure $\s7$ is torsion free if and only if $d\s7=0$. Since $\star \s7=\s7$, this is equivalent to $d^*\s7=0$.
\end{theorem}

\medskip

\noindent
Finally, an important bit of information which we need about the torsion is that it satisfies a ``Bianchi-type identity''. This was first proved by Karigiannis \cite[Theorem 4.2]{karigiannis-spin7} using the diffeomorphism invariance of the torsion tensor and a different proof using the Ricci identity \eqref{eq:ricciiden} was given in \cite[Theorem 3.9]{dle-isometric}.

\begin{theorem}\label{thm:spin7bianchi}
The torsion tensor $T$ satisfies the following ``Bianchi-type'' identity
\begin{align}\label{spin7bianchi}
\del_iT_{j;ab}-\del_jT_{i;ab}=2T_{i;am}T_{j;mb}-2T_{j;am}T_{i;mb}+\frac 14R_{jiab}-\frac 18R_{jimn}\s7_{mnab}.    
\end{align}
\qed
\end{theorem}
\noindent
For fixed $j, i$, we have 
\begin{align*}
R_{jiab}= \pi_7(\Riem)_{jiab}+\pi_{21}(\Riem)_{jiab}
\end{align*}
where, using \eqref{eq:pi7} and \eqref{eq:pi21}, we have
\begin{align*}
\pi_7(\Riem)_{jiab}&=\frac 14R_{jiab}-\frac 18R_{jimn}\Phi_{mnab}\ \ \ \ \text{and}\\
\pi_{21}(\Riem)_{jiab}&=\frac 34R_{jiab}+\frac 18R_{jimn}\Phi_{mnab},
\end{align*}
and hence the Spin(7)-Binachi identity \eqref{spin7bianchi} can also be written as
\begin{align*}
\del_iT_{j;ab}-\del_jT_{i;ab}=2T_{i;am}T_{j;mb}-2T_{j;am}T_{i;mb}+\pi_7(\Riem)_{jiab}.
\end{align*}
\noindent
Using the Riemannian Bianchi identity, we see  that
\begin{align*}
R_{ijkl}\s7_{ajkl}=-(R_{jkil}+R_{kijl})\s7_{ajkl}=-R_{iljk}\s7_{aljk}-R_{ikjl}\s7_{akjl}    
\end{align*}
and hence we have the fact that 
\begin{align}\label{rmspin7}
R_{ijkl}\s7_{ajkl}=0.    
\end{align}
Using this and contracting \eqref{spin7bianchi} on $j$ and $b$ gives the expression for the Ricci curvature of a metric induced by a Spin(7)-structure. Precisely,

\begin{align}\label{ricci}
R_{ij}=4\del_iT_{a;ja}-4\del_aT_{i;ja}-8T_{i;jb}T_{a;ba}+8T_{a;jb}T_{i;ba}    
\end{align}
which also proves that the metric of a torsion free Spin(7)-structure is Ricci-flat, a result originally due to Bonan. Taking the trace of \eqref{ricci} gives the expression of the scalar curvature $R$
\begin{align}
R&=4\del_iT_{a;ia}-4\del_aT_{i;ia}+8|T_8|^2+8T_{a;jb}T_{j;ba}, \nonumber \\   
& \stackrel{\eqref{T8des}}{=} 8\Div T_8 + 8|T_8|^2+8T_{a;jb}T_{j;ba} \label{scalar1}.
\end{align}

\noindent
Since $T=T_8+T_{48}$, the $T_8$ component of torsion can be viewed as a vector field on $M$ and hence we want to look at an alternate expression for the symmetric $2$-tensor $\cL_{T_8}g$ which will be used in \textsection~\ref{sec:ifintro}. We have the following linear algebra result whose proof (and in fact, a more general statement) can be found in \cite[Lemma 4.1]{dgk-flows}.

\begin{lemma} \label{lemma:basic-tool}
Let $V$ and $W$ be finite-dimensional real vector spaces equipped with positive definite inner products.
Let $\iota \colon V \to W$ and $\rho \colon W \to V$ be linear maps. Suppose that there exist $b, c$ \emph{both nonzero}, such that for all $v \in V$ and $w \in W$, we have
\begin{equation} \label{eq:basic-condition}
\mathrm{(i)} \,\, \rho \iota v = b v \quad \text{ and } \quad \mathrm{(ii)} \,\, \langle \rho w, v \rangle = c \langle w, \iota v \rangle.
\end{equation}
Then in fact we have an isomorphism of $W$ with an \emph{orthogonal} direct sum
\begin{equation} \label{eq:basic-tool-result}
W \cong (\ker \rho) \oplus_{\perp} V.
\end{equation}
\qed
\end{lemma}

\noindent
Let $V=\Lambda^1_8$ and $W=\Lambda^1_8\otimes \Lambda^2_7$ in the previous lemma. Define $\iota: \Lambda^1_8\rightarrow \Lambda^1_8\otimes \Lambda^2_7$ by
\begin{align*}
 \iota(X)_{ijk}=X_kg_{ij}-X_jg_{ik}+X_p\s7_{pijk},   
\end{align*}
and $\rho: \Lambda^1_8\otimes \Lambda^2_7\rightarrow \Lambda^1_8$ by
\begin{align*}
\rho(\gamma)_j=\gamma_{i;ji}.    
\end{align*}
Clearly $\iota(X)_{ijk}=-\iota(X)_{ikj}$ and hence \emph{a priori}, $\iota(X)\in \Lambda^1_8\otimes \Lambda^2$ and we now show that in fact, $\iota(X)\in \Lambda^1_8\otimes\Lambda^2_7$. Consider
\begin{align}
 \iota(X)_{ijk}\s7_{abjk}&= \left(X_kg_{ij}-X_jg_{ik}+X_p\s7_{pijk}\right) \Phi_{abjk} \nonumber \\
 &=X_k\s7_{abik}-X_j\s7_{abji}+X_p(6g_{pa}g_{ib}-6g_{pb}g_{ia}-4\s7_{piab}) \nonumber \\
&=2X_k\s7_{abik}+6X_ag_{ib}-6X_bg_{ia}+4X_p\s7_{abip}\nonumber \\
&= 6X_ag_{ib}-6X_bg_{ia}+6X_k\s7_{abik}=-6\iota(X)_{iab}
\end{align}
where we have used \eqref{eq:impiden2} in the second equality. Hence \eqref{eq:27decomp2} implies $\iota(X)\in \Lambda^1_8\otimes \Lambda^2_7$.
We compute
\begin{align}
(\rho \iota (X))_j&= \rho(X_kg_{ij}-X_jg_{ik}+X_p\s7_{pijk})_j=-7X_j. \nonumber
\end{align}
Also,
\begin{align}
\langle \rho \gamma, X\rangle &= \gamma_{i;ji}X_j,   \nonumber 
\end{align}
and, using \eqref{eq:27decomp2},
\begin{align}
 \langle \gamma, \iota X\rangle &= \gamma_{i;jk}(X_kg_{ij}-X_jg_{ik}+X_p\s7_{pijk}) \nonumber \\
 &= \gamma_{i;ik}X_k-X_j\gamma_{i;ji}+X_p\gamma_{i;jk}\s7_{pijk} \nonumber \\
 &= -8X_j\gamma_{i;ji}. \nonumber
\end{align}
Thus, (i) and (ii) in \eqref{eq:basic-condition} are satisfied with $b_k=-7$ and $c_k=-8$ respectively. We deduce from Lemma~\ref{lemma:basic-tool} that
\begin{align}
\Lambda^1_8\otimes \Lambda^2_7 \cong (\ker \rho) \oplus_{\perp} \Lambda^1_8  
\end{align}
with $\ker \rho = \{\gamma_{i;jk}\in \Lambda^1_8\otimes \Lambda^2_7 \mid \gamma_{i;ji}=0\}$ being $48$-dimensional. In particular, since $T_{\Phi}\in \Lambda^1_8\otimes \Lambda^2_7$, we see that 
\begin{equation}\label{T8des}
(T_8)_j = T_{i;ji}    
\end{equation}
and thus
\begin{align}\label{eq:lieT8}
\left(\cL_{T_8}g\right)_{ij}= \del_i(T_8)_j+\del_j(T_8)_i= \del_i(T_{k;jk})+\del_j(T_{k:ik}).   
\end{align}

\noindent
We will also need the expression for the Lie derivative of the Spin(7)-structure in \textsection~\ref{sec:ste}. Let $W\in \Gamma(TM)$. Using \eqref{Tdefneqn} and \eqref{eq:diadefn2}, we have
\begin{align}
(\cL_W\s7)_{ijkl}&= W_p\del_p\s7_{ijkl}+ \del_iW_p\s7_{pjkl}+ \del_jW_p\s7_{ipkl}+ \del_kW_p\s7_{ijpl}+ \del_lW_p\s7_{ijkp} \nonumber \\
&= W_p(T_p\diamond \s7)_{ijkl}+(\del W\diamond \s7)_{ijkl}. \nonumber
\end{align}
We have $\del W = (\del W)_{\text{sym}}+(\del W)_{\text{skew}}= \frac 12 \cL_Wg + (\del W)_7 + (\del W)_{21}$, where $(\del W)_{7}$ and $(\del W)_{21}$ denote the projection of the $2$-form $(\del W)_{\text{skew}}$ onto the $\Omega^2_7$ and $\Omega^2_{21}$ components respectively. Using the fact that $\Omega^2_{21}\diamond \s7=0$, we get
\begin{align}
(\cL_W\Phi)_{ijkl}&=\left(\left(\frac 12 \cL_Wg +  T(W)+(\del W)_7 \right)  \diamond \s7 \right)_{ijkl}. \label{eq:liePhiexp} 
\end{align}
We record this observation in the following

\begin{proposition}\label{prop:infsymmPhi}
A vector field $W$ is an infinitesimal symmetry of $\s7$, i.e., $\cL_W\s7=0$ if and only if
\begin{align}\label{eq:infsymPhi}
 \cL_Wg=0 \ \ \ \ \ \text{and}\ \ \ T(W)=-(\del W)_7.   
\end{align}
\qed
\end{proposition}

\begin{remark}
Proposition~\ref{prop:infsymmPhi} and \eqref{eq:liePhiexp} is essentially \cite[Lemma 2.6]{dle-isometric} where it was proved for arbitrary $H$-structures.    
\end{remark}

\medskip

\noindent
Comparing with the case of manifolds with $\mathrm{G}_2$-structures, the most general flow of $\G2$-structures which are second order quasilinear are described in \cite{dgk-flows} and the authors prove the short-time existence and uniqueness theorem for a large class of flows. Such a family of flows was obtained by explicitly computing all
independent second order differential invariants of $\G2$-structures which are $3$-forms. Since the Riemann curvature tensor $\Riem$ and $\del T$ are the only two second order invariants of a $\G2$-structure, decomposing them into irreducible $\G2$-representations and then picking the linearly independent components (since $\Riem$ and $\del T$ are related by the so called $\G2$-Bianchi identity) which can be made into a $3$-form provides a way to write the most general flow of $\G2$-structures. See \cite[\textsection 4]{dgk-flows} for more details. The existence of a Spin(7)-structure induces a decomposition of tensor bundles into irreducible Spin(7)-representations and one can similarly decompose $\Riem$ and $\del T$ into irreducible Spin(7)-representations and look at the linearly independent components which can be made into an admissible $4$-form. We describe this briefly below without any proofs. 

\medskip

\noindent
Let us denote the $k$-dimensional irreducible Spin(7)-representation by $\mb{k}$. Recall that the space $\mathcal{K}$ of \emph{curvature tensors} on $(M, g)$ is the subspace of $S^2 (\Lambda^2) = \Gamma(\mathrm{S}^2 (\Lambda^2 T^* M))$ of elements satisfying the first Bianchi identity and hence we have the orthogonal decomposition
\begin{align*}
S^2(\Lambda^2)&= \Omega^4 \oplus_{\perp} \mathcal{K}.    
\end{align*}
Since for a manifold with a Spin(7)-structure $\Phi$, we have $\Lambda^2 = \mb{7}\oplus \mb{21}$ (see \textsection~\ref{subsec:formdecomp}), we have
\begin{align*}
S^2(\Lambda^2)=S^2(\mb{7}\oplus \mb{21})=S^2(\mb{7})\oplus \left(\mb{7}\otimes \mb{21}\right) \oplus S^2(\mb{21}).    
\end{align*}

\noindent
One can check the following decomposition of tensor bundles into irreducible Spin(7)-representations.
\begin{align}
\underbrace{\mb{7} \otimes \mb{7}}_{49} &= \mb{1}\oplus \mb{21} \oplus \mb{27} \\
\underbrace{\mb{7} \otimes \mb{21}}_{147} &= \mb{35}\oplus \mb{7} \oplus \mb{105} \\
\underbrace{\mb{21} \otimes \mb{21}}_{441} &= \mb{21}\oplus \mb{189}\oplus \mb{168} \oplus \mb{27} \oplus \mb{1} \oplus \mb{35}\\
\underbrace{S^2(21)}_{231} &= \mb{27} \oplus \mb{168} \oplus \mb{35} \oplus \mb{1}
\end{align}
and hence we have
\begin{align}
S^2(\Lambda^2) &= (\mb{1}\oplus \mb{27}) \oplus (\mb{35}\oplus \mb{7} \oplus \mb{105}) \oplus (\mb{27}\oplus \mb{168}\oplus \mb{35} \oplus \mb{1}).   
\end{align}
Since the space of curvature type tensors $\mathcal{K}$ is orthogonal to the space of $4$-forms $\Omega^4=\mb{1}\oplus \mb{7}\oplus \mb{27} \oplus \mb{35}$, we get the following decomposition
\begin{align*}
 \mathcal{K}&= \underbrace{\mb{1} \oplus \mb{35}}_{\text{Ricci}} \oplus \underbrace{\mb{27} \oplus \mb{105} \oplus \mb{168}}_{\text{Weyl}}.    
\end{align*}
Thus, in the presence of a Spin(7)-structure, the Weyl tensor further decomposes as $W=W_{27}+W_{105}+W_{168}$. Like the $\G2$-case there is again a $W_{27}$, however unlike the $\G2$-case, this $W_{27}$ cannot be used for a flow of Spin(7)-structures as the space of admissible $4$-forms is (pointwise) $43$-dimensional and is $\mb{1}\oplus \mb{7}\oplus \mb{35}$. Consequently, the only contribution from $\Riem$ for a flow of Spin(7)-structures are the traceless Ricci curvature $\Ric_{\circ}$ and the scalar curvature $Rg$.

\medskip

\noindent
In a similar way we look at the possible contributions from $\del T$ for a flow of Spin(7)-structures. Since $T=T_{8}+T_{48}$, hence at every point it lies in the representation $\mb{8}\oplus \mb{48}$. Thus, 
\begin{align*}
\del T&= \del T_8 + \del T_{48} \\ 
& \in (\mb{8}\otimes \mb{8}) \oplus (\mb{8}\otimes \mb{48})\\
&=(\mb{1}\oplus \mb{7} \oplus \mb{21}\oplus \mb{35}) \oplus (\mb{35}\oplus \mb{27}\oplus\mb{189}\oplus\mb{21}\oplus\mb{105}\oplus\mb{7})
\end{align*}
and hence the components which can possibly contribute to a flow of Spin(7)-structures are $\mb{1}, \mb{7}$ and $\mb{35}$. The $\mb{1}$ component from $T$ is the $\Div T_8$ term which can be written in terms of the scalar curvature $R$ and lower order terms as demonstrated in \eqref{scalar1}. Similarly, for the two $\mb{35}$ components of $\del T$ we notice that the Spin(7)-Bianchi identity \eqref{spin7bianchi} is an element of
\begin{align*}
\Lambda^2\otimes \Lambda^2_7 = (\mb{7}\oplus \mb{21})\otimes \mb{7}= (\mb{1}\oplus \mb{21}\oplus \mb{27})\oplus (\mb{35}\oplus \mb{7}\oplus \mb{105})
\end{align*}and so there is one relation between the $\mb{35}$ components of $\Riem$ and $\del T$ which means that only one of the $\mb{35}$s appearing in the decomposition of $\del T$ is independent. Thus, the two $\mb{35}$s appear as a combination from $\cL_{T_8}g$ and $\Ric$. Similarly there is one relation between $\mb{7}$ components of $\Riem$ and $\del T$ and hence only one contribution for the $\mb{7}$ component from $\del T$ is independent which is the $\Div T$ term (since in the Spin(7)-case $\Div T=-\Div T^t$ where $(\Div T^t)_{ij}=\del_mT_{m;ji}$). One can do an explicit description of all the components above from $\Riem$ and $\del T$ as in \cite{dgk-flows}. However, since we are only interested in flows of Spin(7)-structures for which we have all the relevant components, namely, $\Ric, \cL_{T_8}g$ and $\Div T$, we do not pursue the former here and we are currently doing this in collaboration with R. Singhal. 
    

\section{A gradient flow of Spin(7)-structures}
\label{sec: Basics of the Spin(7)-flow}

\subsection{Derivation of the flow}
\label{sec:ifintro}
In this section we define the flow of Spin(7)-structures studied in this paper and prove that it is a negative gradient flow.


\begin{definition}
Let $M^8$ be a compact manifold. The \emph{energy functional} $E$ on the set of Spin(7)-structures on $M$ is
\begin{align}
\label{energyfuncdefn}
    E(\s7)=\frac 12 \int_M |T_{\s7}|^2 \vol_{\s7}     
\end{align}
where $T_{\s7}$ is the torsion of $\s7$ and the norm and the volume form are with respect to the metric induced by $\s7$.
\end{definition}

\noindent
The natural question here is whether, given a Spin(7)-structure $\s7_0$, there is a  ``best'' Spin(7)-structure. An obvious way to study this question is to consider the negative gradient flow of $E$. 

\medskip

\noindent 
The most general flow of Spin(7)-structures \cite{karigiannis-spin7} is given by
\begin{align}\label{gflowdefn}
\frac{\pt \s7}{\pt t}=A\diamond \s7 = (h+X)\diamond \s7   
\end{align}
where $A_{ij}=h_{ij}+X_{ij}$ with $h_{ij}\in S^2$ and $X_{ij}\in \Omega^2_7$. Thus, the most general flow of Spin(7)-structures is given by a time-dependent family of symmetric $2$-tensors and a time-dependent family of $2$-forms lying in the subspace $\Omega^2_7$ with respect to $\Phi(t)$. In this case the evolution of the metric, the inverse of the metric and the volume form are given by 
\begin{align}\label{gflowmetricevol}
\frac{\pt g}{\pt t}=2h_{ij},\ \ \ \ \frac{\pt g^{-1}_{ij}}{\pt t}=-2h_{ij},\ \ \ \ \frac{\pt \vol}{\pt t}=\tr h \vol.    
\end{align}
We start by considering the variation of the torsion $T$ with respect to variation of the Spin(7)-structures.

\begin{proposition}\cite[Thm. 3.4]{karigiannis-spin7}\label{prop:1var}
Let $(\s7_t)_{t\in (-\varepsilon, \varepsilon)}$ be a smooth family of Spin(7)-structures. By \eqref{gflowdefn}, we can write $\left.\frac{\pt}{\pt t}\right|_{t=0}\s7_t=A\diamond \s7 = (h+X)\diamond \Phi$ for some $h\in \Gamma(S^2T^*M)$, $X\in \Omega^2_7$. If $T_t$ is the torsion of $\s7_t$ then
\begin{align}\label{1vareqn}
\left.\frac{\pt}{\pt t}\right|_{t=0}(T_{t})_{m;ab}=(h_{ap}T_{m;pb}-h_{bp}T_{m;pa})+\pi_7(\del_bh_{am}-\del_ah_{bm})+X_{ap}T_{m;pb}-X_{bp}T_{m;pa}+\pi_7(\del_mX_{ab}).  
\end{align}
\end{proposition}

\medskip

\noindent
Let $E$ be the energy functional from \eqref{energyfuncdefn}. The following result motivates our formulation of the flow, cf. Definition \ref{gflowdefn}.


\begin{lemma}\label{lemma:Energyvar}
Let $(\s7_t)_{t\in (-\varepsilon, \varepsilon)}$ be a smooth family of Spin(7)-structures and $\left.\frac{\pt}{\pt t}\right|_{t=0}\s7_t=(h+X)\diamond_t \s7_t$ for some $h\in \Gamma(S^2T^*M)$, $X\in \Omega^2_7$. The gradient of the energy functional $E(\Phi_t)$ from \eqref{energyfuncdefn} is given by  
\begin{align}\label{eq:energyvareq}
\left.\frac{d}{dt}\right|_{t=0} E(\s7_t)&= \int h_{am}\left(\frac 12 R_{am}-(\cL_{T_8}g)_{am}-4T_{b;al}T_{m;lb}+4T_{m;al}T_{b;lb}-T_{a;lb}T_{m;lb} \right.\nonumber\\
& \qquad \qquad \quad \left.+\frac 12|T|^2g_{am} \right)\vol 
 - \int X_{ab} \Div T_{ab}\vol.   
\end{align}
\end{lemma}

\begin{proof}
We use Proposition \ref{prop:1var} and \eqref{gflowmetricevol} to compute
\begin{align}
\left.\frac{d}{dt}\right|_{t=0} E(\s7_t)&=  \left.\frac{d}{dt}\right|_{t=0}\frac 12\int_M(T_t)_{m;ab}(T_t)_{n;cd}\ g^{mn}g^{ac}g^{bd}\vol \nonumber \\
&=\int_M T_{m;ab}\pt_tT_{m;ab}\vol \nonumber \\
& \quad -\int_M(h_{mn}T_{m;ab}T_{n;ab}+h_{ac}T_{m;ab}T_{m;cb}+h_{bd}T_{m;ab}T_{m;ad})\vol \nonumber \\
& \quad + \frac 12 \int_M |T|^2 \tr h \vol. \label{eq:torvar1}
\end{align}    

\noindent
We calculate the first term on the right hand side of \eqref{eq:torvar1}. Using \eqref{eq:pi7} for the $\pi_7$ component and integration by parts, we have
\begin{align*}
 \int T_{m;ab}\pt_tT_{m;ab}\vol &= \int T_{m;ab}[(h_{ap}T_{m;pb}-h_{bp}T_{m;pa})+\pi_7(\del_bh_{am}-\del_ah_{bm})]\vol \\
 & \quad + \int T_{m;ab}[X_{ap}T_{m;pb}-X_{bp}T_{m;pa}+\pi_7(\del_mX_{ab})]\vol \\
 &=\int [2h_{ap}T_{m;pb}T_{m;ab}+2\pi_7(\del_bh_{am})T_{m;ab}]\vol \\
 & \quad + \int [2X_{ap}T_{m;pb}T_{m;ab}+ \pi_7(\del_mX_{ab})T_{m;ab}]\vol \\
 &= \int \left[2h_{ap}T_{m;pb}T_{m;ab}+\left(\frac 12 \del_bh_{am}-\frac 14\del_ih_{jm}\s7_{ijba}\right)T_{m;ab}\right ]\vol \\
 & \quad + \int\left(\frac 14 \del_mX_{ab}T_{m;ab}-\frac 18\del_mX_{ij}\s7_{ijab}T_{m;ab}\right)\vol \\
 \intertext{where we have used the fact that $T_{m;pb}T_{m;ab}$ is symmetric in $a,b$ while $X_{ab}$ is skew-symmetric,}
 &= \int \left(2h_{ap}T_{m;pb}T_{m;ab}+\frac 12 \del_bh_{am}T_{m;ab} - \frac 32 \del_ih_{jm}T_{m;ij}\right)\vol \\
 & \quad + \int \del_mX_{ab}T_{m;ab}\vol \\
 \intertext{which we obtained by using the fact that $T_{m;ab}\in \Omega^1\otimes \Omega^2_7$ and \eqref{eq:27decomp2},}
 &= \int \left(2h_{ap}T_{m;pb}T_{m;ab}+2 \del_bh_{am}T_{m;ab} \right)\vol \\
 & \quad - \int X_{ab} \Div T_{ab}\vol, 
\end{align*}
where we used integration by parts and $(\Div T)_{ab}=\del_mT_{m;ab}$.
 Integrating by parts and then using the Spin(7)-Bianchi identity \eqref{spin7bianchi} gives
\begin{align*}
2\int \del_bh_{am}T_{m;ab}\vol &= -2\int h_{am}\del_bT_{m;ab} \vol \\ \nonumber
&=-2\int h_{am}\left(\del_mT_{b;ab}+2T_{b;al}T_{m;lb}-2T_{m;al}T_{b;lb}+\frac 14R_{mbab}-\frac 18R_{mbkl}\s7_{klab} \right)\vol \nonumber \\
&=-2\int h_{am}\left(-\frac 14 R_{am}+\del_mT_{b;ab} +2T_{b;al}T_{m;lb}-2T_{m;al}T_{b;lb}\right) \vol,
\end{align*}
from which we infer
\begin{align}\label{Evaraux1}
\int T_{m;ab}\pt_tT_{m;ab}\vol &= \int h_{am}\left(\frac 12 R_{am}-2\del_mT_{b;ab}-4T_{b;al}T_{m;lb}+4T_{m;al}T_{b;lb}+2T_{l;mb}T_{l;ab} \right) \vol  \nonumber \\
 & \quad - \int X_{ab} \Div T_{ab}\vol.   
\end{align}

\noindent
Using \eqref{Evaraux1} in \eqref{eq:torvar1} and the expression of $\cL_{T_8}g$ from \eqref{eq:lieT8}, we get
\begin{align}
\left.\frac{d}{dt}\right|_{t=0} E(\s7_t)&= \int h_{am}\left(\frac 12 R_{am}-2\del_mT_{b;ab}-4T_{b;al}T_{m;lb}+4T_{m;al}T_{b;lb}+2T_{l;mb}T_{l;ab} \right. \nonumber \\
& \qquad \qquad \quad \left. -T_{m;lb}T_{a;lb} -T_{l;ab}T_{l;mb}-T_{l;ba}T_{l;bm}+\frac 12|T|^2g_{am} \right)\vol \nonumber \\
& \quad - \int X_{ab} \Div T_{ab}\vol \nonumber \\
&= \int h_{am}\left(\frac 12 R_{am}-(\cL_{T_8}g)_{am}-4T_{b;al}T_{m;lb}+4T_{m;al}T_{b;lb}-T_{a;lb}T_{m;lb} \right.\nonumber\\
& \qquad \qquad \quad \left.+\frac 12|T|^2g_{am} \right)\vol 
 - \int X_{ab} \Div T_{ab}\vol, \nonumber \\
\end{align}
which completes the proof.
\end{proof}

\begin{remark}
A similar calculation for $H$-structures is done in \cite[Prop. 1.44]{FLMS}. The terms coming from $\Phi$ and $g$ are explicit for us as we are looking at a particular $H$-structure.
\end{remark}

\noindent
We are interested in the negative gradient flow of $E(\s7_t)$ and based on the above computations, we propose the following flow of Spin(7)-structures which is the flow of $-2\text{grad}(E(t))$.

\begin{definition}\label{def:ene}
Let $(M^8, \s7_0)$ be a compact manifold with a Spin(7)-structure. Let $(T*T)_{ij}=8T_{b;il}T_{j;lb}-8T_{j;il}T_{b;lb}+2T_{i;lb}T_{j;lb}$. A gradient flow of Spin(7)-structures is the  following initial value problem
\begin{align} 
\label{gfloweqn} 
 \left\{\begin{array}{rl} 
      & \dfrac{\pt \s7}{\pt t} = \left(-\Ric+2(\cL_{T_8}g) + T*T-|T|^2g + 2 \Div T\right) \diamond \s7, \\
      & \s7(0) =\s7_0.
      \tag{GF}
   \end{array}\right.
\end{align}
\end{definition}

\medskip

\noindent
It follows from \eqref{gflowmetricevol}, \eqref{scalar1} and \eqref{T8des} that along \eqref{gfloweqn}, the underlying metric and volume form evolve as
\begin{align}
\pt_tg_{ij}&=-2R_{ij}+4(\cL_{T_8}g)_{ij}+8T_{b;il}T_{j;lb}+8T_{b;jl}T_{i;lb}-8T_{j;il}T_{b;lb}-8T_{i;jl}T_{b;lb}+4T_{i;lb}T_{j;lb}-2|T|^2g_{ij}, \label{gevol} \\
\pt_t(g^{-1})_{ij}&= 2R_{ij}-4(\cL_{T_8}g)_{ij}-8T_{b;il}T_{j;lb}-8T_{b;jl}T_{i;lb}+8T_{j;il}T_{b;lb}+8T_{i;jl}T_{b;lb}-4T_{i;lb}T_{j;lb}+2|T|^2g_{ij}, \label{ginvevol}\\
\pt_t\vol&=-(4\Div T_8+6|T|^2)\vol. \label{volevol}
\end{align}

\begin{remark}
The negative gradient flow in \eqref{gfloweqn} is described (upto the highest order) in terms of the Ricci curvature of the underlying metric, the Lie derivative of the metric in the direction of $T_8$ and $\Div T$. Since $(M^8, \Phi)$ has a spin structure, a similar expression for the negative gradient flow of the spinoral energy functional is obtained in \cite[Prop. 4.14]{AWW} in terms of the covariant derivative of the spinor and $\Div T$.
\end{remark}

\noindent
We discuss the effect of scaling on tensors induced from a Spin(7)-structure. It follows from \cite[\textsection 4.3.1]{dle-isometric} that if $\s7$ is a Spin(7)-structure then so is $\wtd{\s7}=c^4\s7$, $c>0$ constant. In this case, $\wtd{g}=c^2g,\ \wtd{g^{-1}}=c^{-2}g^{-1}$ and $\wtd{\vol}=c^8\vol$. Moreover,
\begin{align*}
\wtd{\del}=\del,\ \ \wtd{T}=c^2T,\ \text{and} \ \wtd{\Ric}=\Ric.    
\end{align*}
The natural parabolic rescaling for a geometric evolution equation involves scaling the time variable $t$ by $c^2t$, when the space variable scales by $c$ and hence, keeping in mind that taking $\diamond$ with $\s7$ involves contraction on one index, we see that each term on the right hand side of \eqref{gfloweqn} indeed has the correct scaling. We record for future reference that for $\wtd{\s7}=c^4\s7$
\begin{align}
|\wtd{\del}^j \wtd{\Riem}|_{\wtd{g}}= c^{-(2+j)}|\del^j \Riem|_g,\ \ \ |\wtd{\del}^j \wtd{T}|_{\wtd{g}}= c^{-(1+j)}|\del^j T|_g.  
\end{align}
In particular, $E(c^4\s7)=c^6E(\s7)$ is a positively homogeneous functional and hence an application of Euler's theorem for positively homogeneous functional shows that the critical points of $E$ are torsion-free Spin(7)-structures, which in fact, are absolute minimizers and thus the flow \eqref{gfloweqn} is capable of detecting torsion-free Spin(7) structures.


\medskip

\section{Short-time existence and uniqueness}\label{sec:ste}
In this section we establish short-time existence and uniqueness of the flow \eqref{gfloweqn} of Spin(7)-structures,
using a modification of the DeTurck’s trick and the explicit computation of the principal symbols of the second order linear differential operators defining the flow. We first calculate the principal symbols of the highest order terms on the right hand side of \eqref{gfloweqn}, i.e., $\Ric, \cL_{T_8}g$ and $\Div T$ in \textsection~\ref{subsec:psymb}. We use these to show that \eqref{gfloweqn} is a \emph{weakly parabolic} PDE and the failure of parabolicity is only due to the diffeomorphism invariance of the tensors involved. We use a modification of the DeTurck's trick from the Ricci flow to prove the short-time existence and uniqueness of solutions in \textsection~\ref{subsec:ste}.

\subsection{Differential operators, ellipticity, and parabolicity} \label{sec:diff-ops}

We give a brief review of parabolic PDEs and the existence and uniqueness of solutions of such equations. Other sources for the discussion below are~\cite[\textsection 3.2]{Chow-Knopf},~\cite[\textsection 5.1]{AH}, ~\cite[\textsection 4]{Topping} and \cite[\textsection 6.1]{dgk-flows} (for the $\mathrm{G_2}$ case).

\medskip

\noindent
Let $E, F$ be vector bundles over a Riemannian manifold $(M,g)$ and let $L \colon \Gamma(E) \to \Gamma(F)$ be a linear differential operator of order $m$. We write $L$, for every $x \in M$, in terms of local frames for $E$ and $M$, as
\begin{equation}
L(\sigma)^b (x) = \sum_{l=0}^m [\hat L_l(x)]_a^{b,i_1,\ldots,i_l} [\nabla^{l}_{i_1,\ldots,i_l} \sigma(x)]^a = \sum_{l=0}^m [\hat L_l(x)]^b (\nabla^l \sigma(x))
\end{equation}
where for each $l =0, 1, \ldots, m$, we write $\nabla^l \sigma \in \Gamma((T^*M)^{\otimes l} \otimes E)$ to denote the $l$-th covariant derivative of $\sigma$, and $\hat L_l \in \Gamma( (TM)^{\otimes l} \otimes \mathrm{Hom}(E,F))$. Here the index $a$ corresponds to a local frame for $E$ and the index $b$ corresponds to a local frame for $F$.

\medskip

\noindent
For any such linear differential operator, we define its principal symbol so that for each $x \in M$ and $\xi \in T^*_x M$, the map
$$ \sigma_{\xi} (L) \colon E_x \to F_x $$
is the linear homomorphism
\begin{equation}
\begin{aligned}
[\sigma_{\xi}(L) (\sigma)]^b & = [\hat L_m (x)]^b (\xi, \ldots, \xi, \sigma), \\
& = [\hat L_m(x)]^{b, i_1, \ldots, i_m}_a \xi_{i_1} \cdots \xi_{i_m} \sigma^a.
\end{aligned}
\end{equation}

\noindent
The principal symbol satisfies the fundamental properties
$$ \sigma_{\xi} (P + Q) = \sigma_{\xi} (P) + \sigma_{\xi} (Q), \qquad \sigma_{\xi} (P \circ Q) = \sigma_{\xi} (P) \circ \sigma_{\xi} (Q), $$
whenever $P$, $Q$ are linear differential operators so that either $P + Q$ or $P \circ Q$ is well defined. We have the following

\begin{definition} \label{def:ellipticity}
A linear differential operator $L \colon \Gamma(E) \to \Gamma(F)$ is called \emph{elliptic} if for any $x \in M$, $\xi \in T^*_x M$, $\xi \neq 0$, the principal symbol $\sigma_{\xi} (L) \colon E_x \to F_x$ is a linear isomorphism.

\noindent
Let $E$ be a vector bundle over $M$ with a fibre metric $\langle \cdot, \cdot \rangle$. Consider a second order linear differential operator $L \colon \Gamma (E) \to \Gamma(E)$. If there is a constant $c > 0$ such that for any $\xi \in T^*_x M$, $\xi \neq 0$ and $v \in E_x$, we have
$$ \langle \sigma_{\xi}(L) (v), v \rangle \geq c |\xi|^2 |v|^2, $$
then $L$ is called \emph{strongly elliptic}.
\end{definition}

\begin{definition} \label{defn:linearization}
Let $E$, $F$ be vector bundles over $M$, let $\mathcal U \subseteq \Gamma(E)$ be open, and let $P \colon \mathcal U \to \Gamma(F)$ be a nonlinear differential operator. The operator $P$ is called elliptic at $v \in \mathcal U$ if the linearization 
\begin{align*}
D_v P & \colon \Gamma(E) \to \Gamma(F), \\
(D_v P) (w) & := \rest{\frac{d}{ds}}{s=0} P(v + sw),
\end{align*}
is an elliptic linear differential operator. 

\noindent
Similarly, if $P \colon \mathcal U \to \Gamma(E)$ is a second order differential operator and $E$ is endowed with a bundle metric $\langle \cdot, \cdot \rangle$, we say that $P$ is strongly elliptic at $\sigma \in \mathcal U$ if its linearization $D_{\sigma} P \colon \Gamma(E) \to \Gamma(E)$ is a strongly elliptic linear differential operator.

\noindent
A nonlinear evolution equation of the form $\frac{\partial}{\partial t} \sigma = P(\sigma)$, where $\sigma \in \mathcal U$, is called \emph{parabolic} at $\sigma$ if $P$ is strongly elliptic at $\sigma$.
\end{definition}

\noindent
The importance of the above definition is due to the following standard result.

\begin{theorem}\label{thm:stdell}
Let $M$ be a Riemannian manifold, let $E$ be a vector bundle over $M$ endowed with a fibre metric $\langle \cdot, \cdot \rangle$, and let $\mathcal U \subseteq \Gamma(E)$ be open. Let $P \colon \mathcal U \to \Gamma(E)$ be a second order quasilinear differential operator, which is strongly elliptic at $\sigma_0 \in \mathcal U$. Then there exists $\epsilon > 0$ and for any $ t\in [0,\epsilon)$ a unique $\sigma(t) \in \mathcal U$, such that
\begin{equation} \label{eq:paraboli_IVP}
\frac{\partial \sigma(t)}{\partial t} = P(\sigma(t)), \qquad \sigma(0) = \sigma_0.
\end{equation}
That is, a nonlinear evolution equation $\frac{\partial}{\partial t} \sigma = P(\sigma)$ which is parabolic at $\sigma_0$ has a unique short time smooth solution with initial condition $\sigma(0) = \sigma_0$.
\qed
\end{theorem}

\subsection{Principal Symbols}\label{subsec:psymb}
Let $\s7$ be a Spin(7)-structure and consider a variation $(\s7_t)_{t\in (\varepsilon, \varepsilon)}$ with $\s7_0=\s7$ and 
\begin{align}
\Dot{\s7}_{ijkl}=\left.\frac{\pt}{\pt t} \right|_{t=0}\s7_{ijkl}=A_{ip}\s7_{pjkl}+A_{jp}\s7_{ipkl}+A_{kp}\s7_{ijpl}+A_{lp}\s7_{ijkp},
\end{align}
where $A_{ij}=h_{ij}+X_{ij}$. Since the variation of the associated Riemannian metric is $\left.\frac{\pt}{\pt t}\right|_{t=0}g_t=2h$ hence the linearization $D\Gamma^r_{ia}$ of the Christoffel symbols satisfies
\begin{equation}
\delta_{qr}D\Gamma^r_{ia}=\nabla_i h_{aq}+\nabla_a h_{iq}-\nabla_q h_{ia},
\end{equation}
where $\nabla$ is the Levi-Civita connection of the metric $g$ associated to the Spin(7)-structure $\Phi$.

\medskip

\noindent
It follows that the principal symbol of $D\Gamma^r_{ia}$, for any non-zero $\xi\in T_x^*M$, is
\begin{equation}
\sigma(\delta_{rq}D\Gamma^r_{ia})(x,\xi)(\dot\Phi)= (\xi_i h_{aq}+\xi_a h_{iq} - \xi_q h_{ia}).
\end{equation}

\noindent
Now we compute the linearization $D(\nabla \s7)_{mijkl}$ of $\nabla_m \Phi_{ijkl}$:
\begin{align*}
D(\nabla \s7)(\Dot{\s7})_{mijkl}= - D\Gamma^r_{mi}(\Dot{\s7}) \s7_{rjkl} - D\Gamma^r_{mj}(\Dot{\s7}) \s7_{irkl}- D\Gamma^r_{mk}(\Dot{\s7}) \s7_{ijrl}- D\Gamma^r_{ml}(\Dot{\s7}) \s7_{ijkr}+ \nabla_m \Dot{\s7}_{ijkl}.
\end{align*}

\medskip

\noindent
The principal symbol of the differential operator $D(\nabla\Phi)$ is:
\begin{align}
\sigma(D(\nabla \s7))(x,\xi)(\dots7)_{mijkl}&=- (\xi_m h_{iq} + \xi_i h_{mq}-\xi_q h_{mi})\s7_{qjkl}- (\xi_m h_{jq} + \xi_j h_{mq}-\xi_q h_{mj})\s7_{iqkl} \nonumber \\
& \quad - (\xi_m h_{kq} + \xi_k h_{mq}-\xi_q h_{mk})\s7_{ijql}- (\xi_m h_{lq} + \xi_l h_{mq}-\xi_q h_{ml})\s7_{ijkq} \nonumber \\
&\quad +\xi_m ( h_{ip}\s7_{pjkl}+h_{jp}\s7_{ipkl}+h_{kp}\s7_{ijpl}+h_{lp}\s7_{ijkp} \nonumber \\
& \qquad \qquad  +X_{ip}\s7_{pjkl}+X_{jp}\s7_{ipkl}+X_{kp}\s7_{ijpl}+X_{lp}\s7_{ijkp}) \nonumber \\
&=- (\xi_i h_{mq}-\xi_q h_{mi})\s7_{qjkl}- (\xi_j h_{mq}-\xi_q h_{mj})\s7_{iqkl} \nonumber \\
& \quad - (\xi_k h_{mq}-\xi_q h_{mk})\s7_{ijql}- (\xi_l h_{mq}-\xi_q h_{ml})\s7_{ijkq} \nonumber \\
&\quad +\xi_m (X_{ip}\s7_{pjkl}+X_{jp}\s7_{ipkl}+X_{kp}\s7_{ijpl}+X_{lp}\s7_{ijkp}).\label{eq:psymdelphi}
\end{align}


\noindent
The linearization $DT_{m;ib}$ of the torsion $T_{m;ib}=\frac{1}{96} \nabla_m \s7_{ijkl} \Phi_{bjkl}$ is given by
\begin{equation}
DT(\dots7)_{m;ib}=\frac{1}{96} D(\nabla \s7)_{mijkl}\s7_{bjkl}+ \textrm{l.o.t.}.
\end{equation}
Hence, its principal symbol is given by
\begin{align*}
\sigma(DT)(x,\xi)(\dot\Phi)_{m;ib}&= \frac{1}{96}\big[ - (\xi_i h_{mq}-\xi_q h_{mi})\s7_{qjkl}-(\xi_j h_{mq}-\xi_q h_{mj})\s7_{iqkl} \nonumber \\
& \qquad \qquad  - (\xi_k h_{mq}-\xi_q h_{mk})\s7_{ijql}- (\xi_l h_{mq}-\xi_q h_{ml})\s7_{ijkq} \nonumber \\
&\qquad \qquad +\xi_m (X_{iq}\s7_{qjkl}+X_{jq}\s7_{iqkl}+X_{kq}\s7_{ijql}+X_{lq}\s7_{ijkq})\big]\s7_{bjkl}
\end{align*}
We use the contraction identities in \eqref{eq:impiden2}, \eqref{eq:impiden3} and \eqref{eq:27decomp2} to compute the principal symbol of $DT$,
\begin{align}\label{eq:psymbtora}
\sigma(DT)(x,\xi)(\dot\Phi)_{m;ib}=\left(\frac 14(\xi_bh_{im}-\xi_ih_{mb}+\xi_jh_{mq}\s7_{ibjq})+\xi_mX_{ib}   \right).
\end{align}

\noindent
Using \eqref{eq:psymbtora} and \eqref{T8des}, the principal symbols of $\Div T_{ib}$ and $(\cL_{T_8}g)_{il}$ are

\begin{align*}
\sigma(D \Div T)(x,\xi)(\dot\Phi)_{ib}&=  \frac 14 (\xi_m\xi_bh_{im}-\xi_m\xi_ih_{mb}+\xi_m\xi_jh_{mq}\s7_{ibjq})+|\xi|^2X_{ib}, \\
\sigma(D \cL_{T_8}g )(x,\xi)(\dot\Phi)_{li}&=  \frac 14 \left(\xi_l\xi_mh_{im}+\xi_i\xi_mh_{lm}-2\xi_i\xi_l\tr h+\xi_l\xi_jh_{mq}\s7_{imjq}+\xi_i\xi_jh_{mq}\s7_{lmjq}\right) \nonumber \\
& \quad +\xi_l\xi_mX_{im}+\xi_i\xi_mX_{lm} \nonumber \\
&= \frac 14 \left(\xi_l\xi_mh_{im}+\xi_i\xi_mh_{lm}-2\xi_i\xi_l\tr h\right) +\xi_l\xi_mX_{im}+\xi_i\xi_mX_{lm}. 
\end{align*}

\noindent
Recall from \eqref{ricci} that the Ricci curvature in terms of the torsion tensor is given by
\begin{align*}
R_{ij}=4\del_iT_{a;ja}-4\del_aT_{i;ja}-8T_{i;jb}T_{a;ba}+8T_{a;jb}T_{i;ba}.    
\end{align*}
From the computations above we obtain that the symbol of $D\Ric$ is given by
\begin{align*}
\sigma(D\Ric)(x,\xi)(\dot\Phi)_{ij}&=4\sigma(D\del T)_{ia;ja}-4\sigma(D\del T)_{ai;ja} \\
&= (\xi_i\xi_ah_{aj}-\xi_i\xi_j \tr h+4\xi_i\xi_aX_{ja})-(-\xi_a\xi_jh_{ia}+|\xi|^2h_{ij}+4\xi_a\xi_iX_{ja})
\end{align*}
which simplifies to
\begin{align*}
\sigma(D\Ric)(x,\xi)(\dot\Phi)_{jk}&= -|\xi|^2h_{ij}-\xi_i\xi_j\tr h+(\xi_i\xi_ah_{aj}+\xi_a\xi_jh_{ia}).
\end{align*}

\noindent
Similarly, the principal symbol of the scalar curvature is
\begin{align*}
\sigma(DR)(x,\xi)(\dot\Phi)&= -2|\xi|^2\tr h+2\xi_j\xi_ah_{ja}.    
\end{align*}
We summarize the above computations in the following
\begin{lemma}
Consider a variation $\pt_t\Phi=(h+X)\diamond \Phi$ of a Spin(7)-structure $\Phi$ with $h$ a symmetric $2$-tensor and $X\in \Omega^2_7(M)$.. For any nonzero $\xi\in T^*_xM$, we have the following principal symbols of first and second order nonlinear differential operators:
\begin{align}
\sigma(DT)(x,\xi)(\dot\Phi)_{m;ib}&=\left(\frac 14(\xi_bh_{im}-\xi_ih_{mb}+\xi_jh_{mq}\s7_{ibjq})+\xi_mX_{ib}   \right), \label{eq:psymbtor}\\
\sigma(D \Div T)(x,\xi)(\dot\Phi)_{ib}&=  \frac 14 (\xi_m\xi_bh_{im}-\xi_m\xi_ih_{mb}+\xi_m\xi_jh_{mq}\s7_{ibjq})+|\xi|^2X_{ib}, \label{eq:psymbdivT}\\
\sigma(D \cL_{T_8}g )(x,\xi)(\dot\Phi)_{li}&= \frac 14 \left(\xi_l\xi_mh_{im}+\xi_i\xi_mh_{lm}-2\xi_i\xi_l\tr h+\right) +\xi_l\xi_mX_{im}+\xi_i\xi_mX_{lm},
\label{eq:psymbsymT}\\
\sigma(D\Ric)(x,\xi)(\dot\Phi)_{jk}&= -|\xi|^2h_{ij}-\xi_i\xi_j\tr h+(\xi_i\xi_ah_{aj}+\xi_a\xi_jh_{ia}), \label{eq:psymbric}\\
\sigma(DR)(x,\xi)(\dot\Phi)&= -2|\xi|^2\tr h+2\xi_j\xi_ah_{ja} \label{eq:psymbscalar}.
\end{align}
\end{lemma}

\subsection{Failure of parabolicity of the flow}\label{subsec:ste}
\noindent
Consider the differential operator $L:\Gamma(S^2T^*M)\oplus \Omega^2_7(M)\rightarrow \Omega^4(M)$
\begin{align*}
 L(h,X)&=h\diamond \Phi + X\diamond \Phi.   
\end{align*}
Clearly, \eqref{gfloweqn} is $\partial_t \Phi(t)=L(-\Ric+2(\cL_{T_8}g) + (T*T)-|T|^2g, 2 \Div T)$. Since $\Ric$ and $T$ are diffeomorphism invariant tensors, $\varphi^*L(\Phi)=L(\varphi^*\Phi)$ for the operator in \eqref{gfloweqn}. For the purposes of short time existence of the flow we are only interested in the highest order term so we instead consider the operator (which we still call $L$)
\begin{align}\label{eq:opL1}
L(\Phi) = L(-\Ric+2\mathcal{L}_{T_8}g, 2\Div T) = (-\Ric+2\mathcal{L}_{T_8}g+2\Div T)\diamond \Phi.    
\end{align}
Moreover, using Proposition~\ref{prop:diaproperties1}, we will view $L:\Gamma(S^2T^*M)\oplus \Omega^2_7\rightarrow \Gamma(S^2T^*M)\oplus \Omega^2_7$. Since the bundle metric $\langle (h_1, X_1), (h_2, X_2) \rangle = \langle h_1, h_2 \rangle + \langle X_1, X_2 \rangle$ on $\Gamma(S^2T^*M) \oplus \Omega^2_7(M)$ is uniformly equivalent to the natural inner product on $\Omega^4_{1\oplus 7\oplus 35}(M)$, the operator $L$ is strongly elliptic if and only if there is a constant $c>0$ such that for any $x \in M$, $\xi \in T^*_x M$, $\xi \neq 0 $, and any $(h,X) \in \Gamma(S^2T^*M)\oplus \Omega^2_7(M)$, we have
\begin{align}\label{eq:strellp}
\langle \sigma_{\xi}(L) (h, X), (h, X) \rangle \geq c | (h, X) |^2 =c ( |h|^2 + |X|^2 ).    
\end{align}
We will see below that the operator $L$ is, in fact, \emph{not} elliptic (and hence \eqref{gfloweqn} not parabolic) but the failure of ellipticity is only due to the diffeomorphism invariance of the tensors in the definition of $L$. As a result, we use a modified DeTurck's trick to prove short-time existence in Theorem~\ref{thm:ste}.

\medskip

\noindent
We first note, using \eqref{eq:psymbric}, \eqref{eq:psymbsymT} and \eqref{eq:psymbdivT}, that
\begin{align}\label{eq:opL2}
\sigma_{\xi}(L)(h,X)_{ij}&= |\xi|^2h_{ij}-\frac 12 (\xi_i\xi_ah_{aj}+\xi_a\xi_jh_{ia})  + 2\xi_i\xi_mX_{jm}+2\xi_j\xi_mX_{im} \nonumber \\
& \quad  +\frac 12 (\xi_m\xi_jh_{im}-\xi_m\xi_ih_{mj}+\xi_m\xi_kh_{mq}\s7_{ijkq})+2|\xi|^2X_{ij}.
\end{align}

\noindent
In order to analyze the kernel of $\sigma_{\xi}(L)$, we define the map
$$ \delta^* \colon \Gamma(TM) \to \Omega^4, \qquad \delta^* W = \cL_W \Phi.$$
From~\eqref{eq:liePhiexp} we have
\begin{equation} \label{eq:delta-star}
\delta^* W = \cL_W \Phi = \left(\tfrac{1}{2} (\cL_W g)+T(W)+(\del W)_7\right) \diamond \Phi.
\end{equation}

\begin{proposition} \label{prop:symbol-Lie-derivative}
Let $\delta^* \colon \Gamma(TM) \to \Omega^4$ be as in~\eqref{eq:delta-star}. For any nonzero $\xi \in T^*_x M$, we have
\begin{equation} \label{eq:LieSymbol}
\begin{aligned}
[\pi_{1+35} \circ \sigma_{\xi} (\delta^*) (W)]_{jk} & = \tfrac{1}{2} (\xi_j W_k + \xi_k W_j), \\
[\pi_7 \circ \sigma_{\xi} (\delta^*) (W)]_{jk}& =\frac 18\xi_jW_k-\frac 18\xi_kW_j-\frac 18\xi_aW_b\Phi_{abjk} , 
\end{aligned}
\end{equation}
and $\sigma_{\xi}(\delta^*) \colon T^*_x M \to \Lambda^4 (T^*_x M)$ is injective. Here $\pi_7$ denotes the identification of $\Omega^4_7$ component of $\sigma_\xi(\delta^*)$ with an element in $\Omega^2_7$.
\end{proposition}
\begin{proof}
The expressions in~\eqref{eq:LieSymbol} follow from~\eqref{eq:delta-star}, because $(\cL_W g)_{jk} = \del_j W_k + \del_k W_j$ and $((\del W)_7)_{jk}= \frac 14((\del W)_{\text{skew}})_{jk}-\frac 18 ((\del W)_{\text{skew}})_{ab}\Phi_{abjk}=\frac 18\del_jW_k-\frac 18\del_kW_j-\frac 18 \del_aW_b\Phi_{abjk}$. Suppose $W \in  \ker \sigma_{\xi}(\delta^*)$. In particular we get $\xi_j W_k + \xi_k W_j = 0$. Multiplying by $\xi_j W_k$ and summing, we obtain
$$ 0 = |\xi|^2 |W|^2 + \langle W, \xi \rangle^2, $$
which implies that $W = 0$ as $\xi\neq 0$, so $\sigma_{\xi}(\delta^*)$ is injective.
\end{proof}

\noindent
Recall that $L(\Phi)$ is invariant under diffeomorphisms, that is, 
\begin{equation*}
L (\varphi^* \Phi) = \varphi^* (L(\Phi)),
\end{equation*}
for any diffeomorphism $\varphi \colon M \to M$. It follows that for any vector field $W \in \Gamma(TM)$, we have
\begin{equation} \label{eq:diffeo_invariance}
\cL_W (L(\Phi)) = D_{\Phi} L(\cL_W \Phi).
\end{equation}
Since $W \mapsto \cL_W (L(\Phi))$ is a first order linear differential operator on $W$, whereas 
$$ W \mapsto D_{\Phi} L(\cL_W \Phi) = (D_{\Phi} L \circ \delta^*) (W)$$ is \emph{a priori} a third order differential operator, it follows that
\begin{equation*}
\sigma_{\xi} (D_{\Phi} L \circ \delta^*) = \sigma_{\xi} (D_{\Phi} L) \circ \sigma_{\xi} (\delta^*) = 0.
\end{equation*}
and hence
\begin{equation} \label{eq:im-delta-star}
\text{im}( \sigma_{\xi} (\delta^*)) = \left\{ \frac 12\xi_iV_j+\frac 12\xi_jV_i+\frac 18\xi_iV_j-\frac 18\xi_jV_i-\frac 18\xi_aV_b\Phi_{abij} : V \in T^*_x M \right \} \subseteq \ker \sigma_{\xi} (D_{\Phi} L).
\end{equation}

\noindent
Hence, by the injectivity of $\sigma_{\xi} (\delta^*)$, the principal symbol of $L$ has
\begin{align}\label{dimpsymbg}
 \text{dim} \ker(\sigma_{\xi}(DL)) \geq 8 = \text{dim}\ M   
\end{align}that is due to diffeomorphism invariance, so $L$ is never an elliptic differential operator. 

\begin{remark}
The relation in \eqref{eq:im-delta-star} can be checked directly by putting $h_{ij}=\frac 12(\xi_iV_j+\xi_jV_i)$ and $X_{ij}=\frac 18(\xi_iV_j-\xi_jV_i-\xi_aV_b\Phi_{abij})$ in \eqref{eq:opL2} (which is, of course, expected).    
\end{remark}

\medskip

\noindent
We now show that the dimension of kernel of $\sigma(L)$ is at most $8$. To do so, we introduce the following two operators. Consider

\begin{equation} \label{eq:B-maps}
\begin{aligned}
B_1 & \colon S^2 (T^*_x M) \to T^*_x M, & B_1 (h)_k & = \xi_a h_{ak} - \tfrac{1}{2} \xi_k \tr h, \\
B_2 & \colon T^*_x M \to \Lambda^2_7(T^*_x M), & B_2 (W)_{ij} & = \frac 18(\xi_iW_j-\xi_jW_i-\xi_aW_b\Phi_{abij}).
\end{aligned}
\end{equation}
We recall that the map $B_1$ is the symbol of the \emph{Bianchi map} $\cS^2 \to \Omega^1$ given by $h \mapsto \del_ih_{ik} - \frac{1}{2} \del_k(\tr h)$. The Ricci curvature $\Ric$ lies in the kernel of the Bianchi map by the twice contracted Riemannian second Bianchi identity. The map $B_2$ is the $\pi_7\circ \sigma_{\xi}(\delta^*)$ map in \eqref{eq:LieSymbol} and is precisely the symbol of the operator $W\mapsto (\del W)_7$. The reason we need the map $B_2$ is because while doing a modified version of the DeTurck's trick, we need to add a term of the form $\cL_W \Phi$ to the right hand side of \eqref{gfloweqn}, and the operator $(\del W)_7, W\in \Gamma(TM)$ shows up in the $\Omega^4_7$ part of $\cL_W \Phi$, by equation~\eqref{eq:liePhiexp}.

\noindent
It is convenient to define the operator $\widetilde B \colon S^2 \oplus \Omega^2_7 \to \Omega^1$ by 
\begin{equation} \label{eq:tildeB1}
\tilde B(h,X)_k = B_1(h)_k - 2 \sigma_{\xi} (D_{\Phi}T_8) (h,X)_k.
\end{equation}

\noindent
We can rewrite the components of $\sigma_{\xi}(DL)$ in terms of the map $\wtd B$. 
\begin{proposition} \label{prop:ricci_like}
Let $L$ be the operator as in~\eqref{eq:opL1}. In terms of the maps $B_1, B_2$ of~\eqref{eq:B-maps} and $\wtd B$ of~\eqref{eq:tildeB1}, the $1+35$ and $7$ parts of the principal symbol of linearization $D_{\Phi} L$ can be expressed as
\begin{equation} \label{eq:psymbP2}
\begin{aligned}
\pi_{1+35} \circ \sigma_{\xi} (DL) (h, X)_{ij} & = |\xi|^2 h_{ij} - \xi_i \wtd B(h,X)_j - \xi_j \wtd B(h,X)_i, \\
\pi_7 \circ \sigma_{\xi} (DL) (h, X)_{ij} & = 2|\xi|^2X_{ij}-8(B_2(\wtd B(h,X)))_{ij}+2\xi_j\xi_mX_{im}-2\xi_i\xi_mX_{jm}+2\xi_k\xi_mX_{qm}\Phi_{kqij}.
\end{aligned}
\end{equation}
\end{proposition}

\begin{proof}
We note that \eqref{eq:opL2} implies
\begin{align}\label{aux1}
 \pi_{1+35}\circ \sigma_{\xi}(DL)(h,X)_{ij}&=   |\xi|^2h_{ij}-\frac 12(\xi_i\xi_ah_{aj}+\xi_a\xi_jh_{ia})  + 2\xi_i\xi_mX_{jm}+2\xi_j\xi_mX_{im}. 
\end{align}
Further, the definitions of the maps $B_1$, $\wtd B$ and equations \eqref{eq:psymbtor}, \eqref{T8des} give
\begin{align}\label{eq:tildeB2}
\wtd B(h,X)_k&= \xi_ih_{ik}-\frac 12\xi_k \tr h -2 \left(\frac 14(\xi_mh_{km}-\xi_k\tr h+\xi_jh_{mq}\s7_{kmjq})+\xi_mX_{km}   \right)   \nonumber \\
& = \frac 12\xi_ih_{ik}-2 \xi_mX_{km}
\end{align}
and hence $\xi_ih_{ik}=2\wtd B(h,X)_{k}+4\xi_mX_{km}$. Substituting this in \eqref{aux1} gives the first equation in \eqref{eq:psymbP2}. We see from \eqref{eq:opL2} that
\begin{align}
\pi_7\circ \sigma_{\xi}(DL)(h,X)_{ij}&=2|\xi|^2X_{ij}+\frac 12(\xi_j\xi_mh_{mi}-\xi_i\xi_mh_{mj}+\xi_k\xi_mh_{mq}\Phi_{kqij}). \label{aux2} \\  
\intertext{Using the expression for $\xi_ih_{ik}=2\wtd B(h,X)_k+4\xi_mX_{km}$ in \eqref{aux2} gives}
&= 2|\xi|^2X_{ij}+ \frac 12 \left(\xi_j(2\wtd B(h,X)_i+4\xi_mX_{im})-\xi_i(2\wtd B(h,X)_j+4\xi_mX_{jm}) \right. \nonumber \\
& \qquad \qquad \qquad\qquad \left.+\xi_k(2\wtd B(h,X)_q+4\xi_mX_{qm})\Phi_{kqij}    \right), \\
\intertext{which on using the definition of the map $B_2$ from \eqref{eq:B-maps} give}
&=2|\xi|^2X_{ij}-8(B_2(\wtd B(h,X)))_{ij}+2\xi_j\xi_mX_{im}-2\xi_i\xi_mX_{jm}+2\xi_k\xi_mX_{qm}\Phi_{kqij} \nonumber
\end{align}
which is the second equation in \eqref{eq:psymbP2}.
\end{proof}

\begin{remark}
The operator $\wtd B$ plays a role similar to the role of the Bianchi operator $B_1$ in the analysis of the principal symbol of the Ricci tensor, for instance in the Ricci flow and for the analysis of the symbol of a large class of flows of $\mathrm{G_2}$-structures as in \cite{dgk-flows}.    
\end{remark}

\medskip

\noindent
Our main goal is to prove that the dimension of kernel of the symbol of the operator $L$ is $8$ which will prove that the diffeomorphism invariance of $\Ric$ and $T$ are the only reason for the failure of parabolicity of \eqref{gfloweqn} and hence we can use a modified version of the DeTurck's trick. To calculate an upper bound on $\text{dim}\ \ker(\sigma(DL))$, we compute the adjoint of the map $\wtd B$.

\noindent
Let $Y\in \Gamma(TM)$. Then
\begin{align}
\langle \wtd B((h,X)), Y\rangle&= \left(\frac 12 \xi_ih_{ik}-2\xi_mX_{km} \right)Y_k \nonumber \\
&=\frac 14h_{ik}(\xi_iY_k+\xi_kY_i)-2X_{km}\left(\frac 18\left(\xi_mY_k-\xi_kY_m\right)-\frac 18 \xi_aY_b\Phi_{abmk} \right) \nonumber,
\end{align}
where we used \eqref{eq:pi7} for $X\in \Omega^2_7(M)$. Thus,
$\wtd B^*:\Omega^1\to S^2\oplus \Omega^2_7$, $\wtd B^*(Y)=(\wtd B_1^*(Y), \wtd B^*_2(Y))$ with
\begin{equation}\label{tildeBadj}
    \begin{aligned}
      \wtd B^*_1(Y)_{ij}&= \frac 14(\xi_iY_k+\xi_kY_i) \\
      \wtd B^*_2(Y)_{ij}&= -\frac 14(\xi_iY_j-\xi_jY_i-\xi_aY_b\Phi_{abij}).
    \end{aligned}
\end{equation}

\begin{lemma}\label{lem:tildeBadj}
The map $\wtd B^* \colon \Omega^1 \to S^2 \oplus \Omega^2_7$ is injective. Consequently, $\dim (\ker \wtd B) = 35$.    
\end{lemma}

\begin{proof}
Let $Y\in \ker \wtd B^*$. Then $\wtd B^*_1(Y)=0$ and hence
\begin{align*}
0&= |\wtd B^*_1(Y)|^2 \\
&= \frac {1}{16}(\xi_iY_k+\xi_kY_i)(\xi_iY_k+\xi_kY_i)\\
&= \frac {1}{8}(|\xi|^2|Y|^2 + \langle \xi, Y\rangle^2)
\end{align*}
and since $\xi\neq 0$, $Y=0$. This proves that $\wtd B^*$ is injective and hence $\text{dim}\ \text{im}(\wtd B^*)=8$. Since $S^2\oplus \Omega^2_7= \ker \wtd B^* \oplus \text{im}\  \wtd B^*$, we get that $\text{dim} (\ker \wtd B)=36+7-8=35$.
\end{proof}

\noindent
We prove our main result in this section on the $\text{dim}$ $\ker (\sigma(DL))$.

\begin{proposition} \label{prop:dimL}
Consider the operator $L$ from \eqref{eq:opL1}. Then 
$\ker(\sigma_{\xi}(DL))=\text{im} (\sigma_{\xi}(\delta^*))$ and hence $\text{dim} \ker(\sigma(DL))=8=\text{dim}\  M$ and the failure of parabolicity of \eqref{gfloweqn} is only due to diffeomorphism invariance of the tensors involved.
\end{proposition}

\begin{proof}
We observe from Proposition~\ref{prop:ricci_like}, equation \eqref{eq:psymbP2} that
\begin{align}\label{dimLaux1}
\sigma_{\xi}(DL)\mid_{\ker \wtd B}(h,X)_{ij}&= |\xi|^2h_{ij}+2|\xi|^2X_{ij}+2\xi_j\xi_mX_{im}-2\xi_i\xi_mX_{jm}+2\xi_k\xi_{m}X_{qm}\Phi_{kqij}
\end{align}
and hence if $(h,X)\in \ker (\sigma_{\xi}(DL))\cap \ker \wtd B$ then $h=0$ and
\begin{align}\label{dimLaux2}
|\xi|^2X_{ij}+\xi_j\xi_mX_{im}-\xi_i\xi_mX_{jm}+\xi_k\xi_{m}X_{qm}\Phi_{kqij}&=0.  
\end{align}
Multiplying both sides by $\xi_i$ gives
\begin{align}
&|\xi|^2\xi_iX_{ij}+\xi_i\xi_j\xi_mX_{im}-|\xi|^2\xi_mX_{jm}+\xi_i\xi_k\xi_mX_{qm}\Phi_{kqij}=0 \nonumber
\intertext{which on using the fact that $\xi_i\xi_j$ is symmetric in $i$ and $j$ whereas $X_{ij}$ and $\Phi_{ijkl}$ are skew-symmetric in $i, j$ and $\xi\neq 0$ gives}
&\xi_{i}X_{ij}=0, \nonumber
\end{align}
and so $X(\xi)=0$. Therefore, from \eqref{dimLaux2}, we get $|\xi|^2X_{ij}=0$ and hence $X=0$. Thus, 
\begin{align}\label{dimLaux3}
\ker (\sigma_{\xi}(DL))\cap \ker \wtd B=\{0\}.    
\end{align}
Since $\text{dim} \ker(\wtd B)=35$ by Lemma~\ref{lem:tildeBadj}, we get
\begin{align}\label{dimpsymbl}
\text{dim} \ker(\sigma_{\xi}(DL))\leq 36+7-35=8,   
\end{align}
which in combination with \eqref{dimpsymbg} gives $\text{dim} \ker(\sigma_{\xi}(DL))=8$.
\end{proof}

\subsection{A modified DeTurck's trick}\label{subsec:mdtt}
In this section we prove that, given a background Spin(7)-structure $\wtd \Phi$, it is always possible to modify the operator $L$ from \eqref{eq:opL1} to an operator which is strongly elliptic at $\wtd \Phi$, that is an operator $\wtd L$ whose symbol satisfies
$$ \langle [\sigma_{\xi} (D_{\tilde \Phi} \wtd L)] (h, X), (h, X) \rangle \geq c |\xi|^2 | (h,X) |^2 = c |\xi|^2(|h|^2+|X|^2)$$for some constant $c > 0$. We do this by doing a modification of the DeTurck's which was originally formulated to give an alternative proof of the short-time existence and uniqueness of solutions to the Ricci flow by DeTurck \cite{DeTurck}. 

\medskip

\noindent
Let $\wtd \Phi$ be a fixed Spin(7)-structure on $M$, for instance, one can take the initial Spin(7)-structure when running the flow. Motivated by the definition of the map $\wtd B$ in \eqref{eq:tildeB1}, we define the vector field $W(\Phi, \wtd \Phi)$ on $M$ by
\begin{align}\label{eq:detvfdef}
W^k&=g^{ij}\left(\Gamma^k_{ij}-\wtd \Gamma^k_{ij} \right)-4(T_{8})^k=\wtd W^k-4(T_8)^k,
\end{align}
where $\wtd g$ is the Riemannian metric induced by $\wtd \Phi$ and $\wtd \Gamma$ is its Christoffel symbols. The vector field $\wtd W$ is the same vector field which is used in the DeTurck's trick for the Ricci flow and the vector field $-4T_8$ is the extra term which we need for the DeTurck's trick in the Spin(7)-case. We define the modified operator $\wtd L$ as
\begin{align}
\wtd L(\Phi)&= L(\Phi)+\cL_{W(\Phi, \wtd \Phi)} \Phi    
\end{align}
where $L$ is the operator in \eqref{eq:opL1}. Using \eqref{eq:liePhiexp} we have
\begin{align*}
\cL_{W}\Phi=\cL_{\wtd W-4T_8}\Phi = \left( \frac 12 \cL_{\wtd W}g-2\cL_{T_8}g + T(\wtd W-4T_8) +(\del \wtd W-4\del T_8)_{7}  \right)\diamond \Phi.    
\end{align*}
Since we only need the highest order terms for the purpose of short-time existence and uniqueness, we neglect the $T(\wtd W-4T_8)$ term above and consequently, the operator $\wtd L$ is given by
\begin{align}\label{eq:optildeL}
\wtd L(\Phi)= \left(-\Ric + \frac 12\cL_{\wtd W}g + 2\Div T + (\del \wtd W-4\del T_8)_7  \right)\diamond \Phi.    
\end{align}
We claim that the operator $\wtd L$ is strongly elliptic. We need to calculate the principal symbol of the linearization of $\wtd L$. It is well known (for example, see \cite[\textsection 3.2]{Chow-Knopf}) that the linearization of $\wtd W$, up to lower order terms, is
\begin{align*}
(D_{\wtd \Phi} \wtd W)(h,X)=2\left(\Div_{\wtd g}h-\frac 12 \del \tr_{\wtd g}h \right).    
\end{align*}
The factor of 2 is here because the variation of metric in the Spin(7) case is given by $2h$. Thus,
\begin{align*}
\sigma_{\xi}\left(\frac 12 D\cL_{\wtd W}g \right)(h,X)_{ij}&=\sigma_{\xi} \left( \frac 12\left(\del_i(2 \Div h-\del \tr h)_j+\del_j(2 \Div h-\del \tr h)_i \right) \right)\\
&=\xi_i\xi_mh_{mj}+\xi_j\xi_mh_{mi}-\xi_i\xi_j\tr h,    
\end{align*}
which on using \eqref{eq:psymbric} gives
\begin{align}\label{eq:psymtilric}
\sigma_{\xi}\left(D\left(-\Ric+\frac 12 \cL_{\wtd W}g\right)\right)&=|\xi|^2h_{ij}.    
\end{align}
We now calculate the symbol of the remaining terms in \eqref{eq:optildeL}. Since
\begin{align*}
(\del \wtd W_7)_{ij}&= \frac 18 \del_i\wtd W_j-\frac 18 \del_j\wtd W_i-\frac 18\del_a\wtd W_b\Phi_{abij}    
\end{align*}
so 
\begin{align*}
(D_{\wtd \Phi}  \del \wtd W_7)_{ij}&= \frac 18\del_i(2\Div h-\tr h)_j-\frac 18\del_j(2\Div h-\tr h)_i-\frac 18\del_a(2\Div h-\tr h)_b\Phi_{abij}    
\end{align*}
and hence
\begin{align}\label{eq:symbaux1}
\sigma_{\xi}(D_{\wtd \Phi}  \del \wtd W_7)_{ij}&= \frac 18 \xi_i(2\xi_mh_{mj}-\xi_j\tr h) -\frac 18 \xi_j(2\xi_mh_{mi}-\xi_i\tr h)     -\frac 18 \xi_a(2\xi_mh_{mb}-\xi_b\tr h)\Phi_{abij} \nonumber \\
&=\frac 14 \xi_i\xi_mh_{mj}-\frac 14 \xi_j\xi_mh_{mi}-\frac 14 \xi_a\xi_mh_{mb}\Phi_{abij}.
\end{align}

\noindent
Similarly,
\begin{align*}
-4((\del T_8)_7)_{ij}&= -\frac 12 \del_i(T_{m;jm})+\frac 12 \del_j(T_{m;im})+\frac 12 \del_a(T_{m;bm})\Phi_{abij}    
\end{align*}
which on using \eqref{eq:psymbtor} gives
\begin{align}
-4(\sigma_{\xi}(D_{\wtd \Phi}(\del T_8)_7))_{ij}&=  -\frac 12 \xi_i\left(\frac 14(\xi_mh_{jm}-\xi_j\tr h+\xi_kh_{mq}\s7_{jmkq})+\xi_mX_{jm}\right)\nonumber \\
& \quad +\frac 12 \xi_j\left(\frac 14(\xi_mh_{im}-\xi_i\tr h+\xi_kh_{mq}\s7_{imkq})+\xi_mX_{im}\right) \nonumber \\
& \quad +\frac 12 \xi_a\left(\frac 14(\xi_mh_{bm}-\xi_b\tr h+\xi_kh_{mq}\s7_{bmkq})+\xi_mX_{bm}\right)\Phi_{abij} \nonumber \\
&= -\frac 18 \xi_i\xi_mh_{jm}+\frac 18 \xi_j\xi_mh_{im}+\frac 18 \xi_a\xi_mh_{bm}\Phi_{abij} \nonumber \\
& \quad - \frac 12\xi_i\xi_mX_{jm}+\frac 12\xi_j\xi_mX_{im}+\frac 12\xi_a\xi_mX_{bm}\Phi_{abij}.\label{eq:symbaux2}
\end{align}

\noindent
Using \eqref{eq:psymtilric}, \eqref{eq:psymbdivT}, \eqref{eq:symbaux1} and \eqref{eq:symbaux2} we compute the symbol of $\wtd L$ as
\begin{align}
\sigma_{\xi}(D\wtd L(\Phi)(h,X))_{ij}&= |\xi|^2h_{ij} + \frac 12 (\xi_m\xi_jh_{im}-\xi_m\xi_ih_{mj}+\xi_m\xi_ah_{mb}\s7_{ijab})+2|\xi|^2X_{ij}  \nonumber \\
& \quad + \frac 14 \xi_i\xi_mh_{mj}-\frac 14 \xi_j\xi_mh_{mi}-\frac 14 \xi_a\xi_mh_{mb}\Phi_{abij} -\frac 18 \xi_i\xi_mh_{jm}+\frac 18 \xi_j\xi_mh_{im}+\frac 18 \xi_a\xi_mh_{bm}\Phi_{abij} \nonumber \\
& \quad - \frac 12\xi_i\xi_mX_{jm}+\frac 12\xi_j\xi_mX_{im}+\frac 12\xi_a\xi_mX_{bm}\Phi_{abij} \nonumber \\
&= |\xi|^2h_{ij}+2|\xi|^2X_{ij} -\frac 38 \xi_i\xi_mh_{mj}+\frac 38 \xi_j\xi_mh_{mi}+\frac 38 \xi_a\xi_mh_{mb}\Phi_{abij} \nonumber \\
& \quad -\frac 12 \xi_i\xi_mX_{jm}+\frac 12 \xi_j\xi_mX_{im}+\frac 12 \xi_a\xi_mX_{bm}\Phi_{abij}.\label{eq:symbaux3}
\end{align}

\medskip

\noindent
Using \eqref{eq:symbaux3} we calculate 
\begin{align}
\left \langle \sigma_{\xi}(D\wtd L(\Phi)(h,X)), (h,X) \right \rangle &=|\xi|^2|h|^2+2|\xi|^2|X|^2-3\xi_i\xi_mh_{mj}X_{ij}-4\xi_i\xi_mX_{jm}X_{ij} \nonumber \\
&= |\xi|^2|h|^2+2|\xi|^2|X|^2 +3\langle h(\xi), X(\xi) \rangle + 4|X(\xi)|^2 \nonumber \\
& \geq |\xi|^2|h|^2+2|\xi|^2|X|^2 -3|h(\xi)||X(\xi)|+4|X(\xi)|^2 \nonumber \\
\intertext{which on using Young's inequality on the 3rd term gives}
& \geq |\xi|^2|h|^2+2|\xi|^2|X|^2 - \frac{1}{2}|h(\xi)|^2-\frac 92|X(\xi)|^2+4|X(\xi)|^2 \nonumber \\
\intertext{and we use $|h(\xi)|^2\leq |h|^2|\xi|^2$ and $|X(\xi)|^2\leq |X|^2|\xi|^2$ to get}
& \geq |\xi|^2|h|^2+2|\xi|^2|X|^2 - \frac 12 |h|^2|\xi|^2 - \frac 12 |X|^2|\xi|^2 \nonumber \\
& \geq \frac 12|\xi|^2(|h|^2+|X|^2)=\frac 12|(h,X)|^2. \label{eq:symbaux4}
\end{align}

\medskip

\noindent
The above computations leading to \eqref{eq:symbaux4} and \eqref{eq:strellp} prove the following 
\begin{proposition}\label{prop:strell}
 Let $(M^8, \wtd \Phi)$ be an $8$-manifold with a Spin(7)-structure $\wtd \Phi$ and let $\wtd L$ be the operator 
 \begin{align}
 \wtd L(\Phi)=L(\Phi)+ \cL_{W(\Phi, \wtd \Phi)} \Phi =  \left(-\Ric + 2\cL_{T_8}g+ \frac 12\cL_{W}g + 2\Div T + (\del \wtd W-4\del T_8)_7  \right)\diamond \Phi,    
 \end{align}
 where $W(\Phi, \wtd \Phi)^k=g^{ij}\left(\Gamma^k_{ij}-\wtd \Gamma^k_{ij}-4(T_8)^k\right)$. Then $\wtd L(\Phi)$ is strongly elliptic at $\wtd \Phi$. 
 \qed
\end{proposition}

\subsection{Short-time existence and uniqueness}\label{subsec:stef}
In this section, we prove the main theorem of the paper by using the modified DeTurck's trick whose details were given in \textsection~\ref{subsec:mdtt}.

\begin{theorem}\label{thm:ste}
Let $(M^8, \Phi_0)$ be a compact $8$-manifold with a Spin(7)-structure $\Phi_0$ and consider the negative gradient flow \eqref{gfloweqn} of the natural energy functional $E$ in \eqref{energyfuncdefn}. Then there exists $\varepsilon >0$ and a unique smooth solution $\Phi(t)$ of \eqref{gfloweqn} for $t\in [0, \varepsilon)$ with $\varepsilon=\varepsilon(\Phi_0)$.
\end{theorem}

\begin{proof}
Let $\wtd \Phi=\Phi_0$ and define $W(\Phi, \wtd \Phi)$ as in \eqref{eq:detvfdef}. Since the terms $8T_{b;al}T_{m;lb}-8T_{m;al}T_{b;lb}+2T_{a;lb}T_{m;lb}-|T|^2g$ are at most first order in $\Phi$, we deduce from Proposition~\ref{prop:strell} that the linearization of the operator
\begin{align*}
\wtd L(\Phi)= \left(-\Ric+2(\cL_{T_8}g) + 8T_{b;al}T_{m;lb}-8T_{m;al}T_{b;lb}+2T_{a;lb}T_{m;lb}-|T|^2g + 2 \Div T\right)\diamond \Phi + \cL_{W}\Phi  
\end{align*}
is strongly elliptic. We obtain from standard parabolic theory (Theorem~\ref{thm:stdell}) that there is a unique smooth solution $\wtd \Phi(t)$ for $t\in [0, \varepsilon)$ of the initial value problem
\begin{align*}
\frac{\partial}{\partial t} \Bar{\Phi}(t)&= \wtd L(\Bar{\Phi}(t)), \\
\Bar{\Phi}(0)&= \Phi_0.
\end{align*}
Let $\varphi_t:M\to M,\ t\in [0, \varepsilon)$ be the one-parameter family of diffeomorphisms defined by 
\begin{equation*}
\begin{aligned}
\frac{\partial}{\partial t} \varphi_t & = - W(\Bar{\Phi}(t), \wtd \Phi) \circ \varphi_t, \\
\varphi_0 & = \Id_M.
\end{aligned}
\end{equation*}
Since $M$ is compact, the family of diffeomorphisms $\varphi_t$ exists by \cite[Lemma 3.15]{Chow-Knopf} as long as the solutions $\Bar{\Phi}(t)$ exists. The family $\Phi(t)=\varphi^*_t(\Bar{\Phi}(t))$ then satisfies
\begin{align*}
\dfrac{\partial \Phi(t)}{\partial t}=\partial_t(\varphi^*_t(\Bar{\Phi}(t)))&=\varphi^*_t\left(\cL_{-W(t)}\Bar{\Phi}(t) +\partial_t\Bar{\Phi}(t)\right) \\
&= \varphi^*_t\left(\left(\cL_{-W(t)}\Bar{\Phi}(t)-\overline{\Ric}+2\cL_{\Bar{T}_8}\Bar{g}+ \overline{\text{l.o.t.}}+2 \overline{\Div} \Bar{T}\right)\diamond_{\Bar{\Phi}(t)} \Bar{\Phi}(t)+\cL_{W(t)}\Bar{\Phi}(t)  \right)\\
&= \left(   -\overline{\Ric}+2\cL_{\Bar{T}_8}\Bar{g}+ \overline{\text{l.o.t.}}+2 \overline{\Div} \Bar{T}\right) \diamond_{(\varphi^*_t(\Bar{\Phi}(t)))} (\varphi^*_t(\Bar{\Phi}(t)))\\
&= \left(-\Ric+2(\cL_{T_8}g) + 8T_{b;al}T_{m;lb}-8T_{m;al}T_{b;lb}+2T_{a;lb}T_{m;lb}-|T|^2g + 2 \Div T\right)\diamond \Phi
\end{align*}
with $\Phi(0)=\Id^*(\Bar{\Phi}(0))=\Phi_0$ and $\overline{\text{l.o.t}}=8\Bar{T}_{b;al}\Bar{T}_{m;lb}-8\Bar{T}_{m;al}\Bar{T}_{b;lb}+2\Bar{T}_{a;lb}\Bar{T}_{m;lb}-|\Bar{T}|^2_{\Bar{g}}\Bar{g}$. This proves the short-time existence of solutions to \eqref{gfloweqn}. 

\medskip

\noindent
We now prove uniqueness, by using the uniqueness of solutions to the harmonic map heat flow. Suppose $\Phi_i(t), i=1,2$ are solutions to \eqref{gfloweqn} with the same initial condition. Let $(F_i)_t:M\to M$ be a one-parameter family of diffeomorphisms given as
\begin{equation}
\begin{aligned}
\frac{\partial}{\partial t}(F_i)_t&=-4(T_8)_{\Phi_i(t)}\circ (F_i)_t,\\
(F_i)_0&= \Id_{M}.
\end{aligned}    
\end{equation}
Using \eqref{eq:liePhiexp}, we see that the family $\Bar{\Phi}_i(t)=(F_i)^*_t\Phi_i(t)$ solves
\begin{equation}
\begin{aligned}
 \frac{\pt}{\pt t}\Bar{\Phi}_i(t)&=\left(-\Ric_{\Bar{g}_i(t)}-4(\Bar{\del} \Bar{T}_8)_{7}+2\Div \Bar{T}_{8} + \overline{\text{l.o.t}}\right)\diamond \Bar{\Phi}_i(t) \\
 \Bar{\Phi}_{i}(0)=\Phi_0.
\end{aligned}    
\end{equation}

\noindent
If we set $\hat{\Phi}_{i}(t)=((\varphi_{i})^{-1}_t)^*\Bar{\Phi}_i(t)$ by defining the diffeomorphisms $\varphi_i(t)$ as $\pt_t(\varphi_i)_t=-\wtd W(\hat{g}_i(t), \wtd g)$ then it is known \cite[\textsection 4.3]{Chow-Knopf} that $\varphi_i(t):M\to M,\ t\in [0, \varepsilon_i^{'})$ is a solution to
\begin{equation}
 \begin{aligned}
 \frac{\pt}{\pt t}\varphi_i(t)&= \Delta_{g_i(t), g_0} (\varphi_i)(t),\\
 \varphi_i(0)&=\Id_M
 \end{aligned}   
\end{equation}
which is the harmonic map heat flow from $(M^8, g_0)$ to $(M^8, g_{i}(t))$ with the initial value being the identity map. Again using \eqref{eq:liePhiexp}, we see that the maps $\hat{\Phi}_i(t)$ satisfy
\begin{equation}\label{mthmaux4}
\begin{aligned}
\frac{\pt}{\pt t}\hat{\Phi}_i(t)&= \left(-\Ric_{\hat{g}_i(t)}+\frac 12 \cL_{\wtd W(\hat{g}_i(t), \wtd g)}\hat{g}_i(t)+(\hat{\del}\wtd W-4\hat{\del}\hat{T}_8)_{7}+2\Div \hat{T}_8 + \widehat{\text{l.o.t.} } \right) \diamond \hat{\Phi}_i(t)\\
\hat{\Phi}_i(0)&=\Phi_0.
\end{aligned}    
\end{equation}
for $t\in [0, \varepsilon_{i}^{'}).$ Since we have proved above that the operator in \eqref{mthmaux4} is parabolic, we know from the standard theory of uniqueness of parabolic equations that $\hat{\Phi}_1(t)=\hat{\Phi}_2(t)$ for all $t\in [0, \varepsilon^{'})$ with $\varepsilon^{'}=\text{min}\{\varepsilon_1^{'}, \varepsilon_2^{'}\}$ and as a result $\hat{g}_1(t)=\hat{g}_2(t)$ which gives $\varphi_1(t)=\varphi_2(t)$. Consequently
\begin{align*}
 \Bar{\Phi}_1(t)=(\varphi_1)^*_t\hat{\Phi}_1(t)=(\varphi_2)^*_t\hat{\Phi}_2 (t)=\Bar{\Phi}_2(t),\ \ \ \ \ \ \ \ \text{for\ all}\ \ \ t\in[0, \varepsilon^{'}).  
\end{align*}
Finally, we have
\begin{align*}
0=\frac{\pt \Id_M}{\pt t}&=\frac{\pt}{\pt t}\left((F_i)_t\circ (F_i)^{-1}_t \right)\\
&= \frac{\pt (F_i)_t}{\pt t}\circ (F_i)^{-1}_t + ((F_i)_t)_*\left(\frac{\pt (F_i)^{-1}_t}{\pt t} \right)\\
&= -4(T_8)_{\Phi_i(t)}+ ((F_i)_t)_*\left(\frac{\pt (F_i)^{-1}_t}{\pt t} \right)
\end{align*}
which gives
\begin{align*}
\left(\frac{\pt (F_i)^{-1}_t}{\pt t} \right)= 4((F_i)^{-1}_t)_*(T_8)_{\Phi_i(t)}=4(T_8)_{(F_i)^{*}_t \Phi_i(t)}\circ (F_i)^{-1}_t=4(T_8)_{\Bar{\Phi}_i(t)}\circ (F_i)^{-1}_t  . 
\end{align*}
But since we proved above that $\Bar{\Phi}_1(t)=\Bar{\Phi}_2(t),$ we get $(F_1)^{-1}_t=(F_2)^{-1}_t$ and hence $\Phi_1(t)=\Phi_2(t)$ for all $t\in [0, \varepsilon^{'})$ which proves the uniqueness of solutions to \eqref{gfloweqn} and completes the proof of the theorem.
\end{proof}

\begin{remark}
A particularly interesting flow of Spin(7)-structures is the "coupling" of the Ricci flow of the metric and the isometric/harmonic flow of Spin(7)-structures
$$ \frac{\partial}{\partial t} \Phi(t) = (- \Ric + \Div T) \diamond \Phi.$$
This flow induces \emph{precisely} the Ricci flow $\frac{\partial}{\partial t} g = - 2 \Ric$ on the metric, and the only other thing it does to the Spin(7)-structure $\Phi$ is to deform it by the isometric flow which was studied in detail in \cite{dle-isometric}. This ``coupling'' of Ricci flow with the isometric flow has good short-time existence and uniqueness. This can be seen by going through the modified DeTurck's trick in \textsection~\ref{subsec:mdtt} and choosing $W=\wtd W$ and then checking that the resulting modified flow is strictly parabolic. The isometric flow of Spin(7)-structures has many good properties, in particular there is an almost monotonicity formula for the solutions to the flow (see \cite{dle-isometric}). It would be interesting to study whether the Ricci flow coupled with the harmonic flow has similar analytic properties. 
\end{remark}

\section{Solitons}\label{sec:solitons}

\noindent
Let $M^8$ be an $8$-manifold. A \emph{soliton} for \eqref{gfloweqn} is a triple $(\Phi, Y, \lambda)$ with $Y\in \Gamma(TM)$ and $\lambda\in \bR$ such that
\begin{align}\label{sol1}
\left(-\Ric+2(\cL_{T_8}g) + T*T-|T|^2g + 2 \Div T\right) \diamond \s7=\lambda \Phi+\cL_Y\Phi
\end{align}
where $(T*T)_{ij}= 8T_{b;il}T_{j;lb}-8T_{j;il}T_{b;lb}+2T_{i;lb}T_{j;lb}$. Those Spin(7)-structures which satisfy \eqref{sol1} are special as they give self-similar solutions to \eqref{gfloweqn}. Recall that a self-similar solution to \eqref{gfloweqn} is a solution of the form 
\begin{align*}
\Phi(t) = \lambda(t)^4 \varphi(t)^* \Phi(0)
\end{align*}
where $\lambda(t)$ are time-dependent scalings and $\varphi(t):M\to M$ are a one-parameter family of diffeomorphisms. Note that the power of $4$ on $\lambda(t)$ is just for convenience in calculations. A straightforward calculation (for instance see \cite[Lemma 2.17]{dgk-isometric}, \cite[Prop. 2.11]{dle-isometric} or \cite[Prop. 1.55]{FLMS}) shows that the solitons for \eqref{gfloweqn} and self-similar solutions are in one-to-one correspondence. 

\medskip

\noindent
We say a soliton $(\Phi, Y, \lambda)$ is expanding if $\lambda>0$; steady if $\lambda=0$; and shrinking if $\lambda <0$.

\medskip

\noindent
We now derive the condition satisfied by the metric $g$ induced by $\Phi$ and $\Div T$ when $(\Phi, Y, \lambda)$ is a soliton, which we expect to have further use.

\medskip

\begin{proposition}\label{prop:gsoliton}
Let $(\Phi, Y, \lambda)$ be a solitons as defined in \eqref{sol1}. Then the induced metric $g$ satisfies
\begin{align}\label{gsoliton1}
-R_{ij}+2(\cL_{T_8}g)_{ij}+4T_{b;il}T_{j;lb}+4T_{b;jl}T_{i;lb}-4T_{j;il}T_{b;lb}-4T_{i;jl}T_{b;lb}+2T_{i;lb}T_{j;lb}-|T|^2g_{ij}=\frac{\lambda}{4}g_{ij}+\frac 12(\cL_Yg)_{ij},
\end{align}
and $\Div T$ satisfies
\begin{align}\label{gsoliton2}
(\Div T)_{ij}+4T_{b;il}T_{j;lb}-4T_{b;jl}T_{i;lb}-4T_{j;il}T_{b;lb}+4T_{i;jl}T_{b;lb}= \frac 12\left(T(Y)+(\del Y)_7 \right)_{ij}.
\end{align}
\end{proposition}
\begin{proof}
The definition of the $\diamond$ operator in \eqref{eq:diadefn1} implies $\Phi=\left(\dfrac g4 \diamond \Phi \right)$. Moreover, \eqref{eq:liePhiexp} gives $\cL_Y\Phi=\left(\frac 12 \cL_Yg+T(Y)+(\del Y)_7 \right)\diamond \Phi$. Putting these expressions in \eqref{sol1} while decomposing the tensor $T*T$ into symmetric and skew-symmetric parts and using the fact that $\Omega^4_{1+ 35}$ is orthogonal to $\Omega^4_7$ gives \eqref{gsoliton1} and \eqref{gsoliton2}. 
\end{proof}

\medskip

\noindent
We can prove the following non-existence theorem for compact expanding solitons of \eqref{gfloweqn}.

\begin{proposition}\label{prop:solmainprop}
 Let $(M^8, \Phi, Y, \lambda)$ satisfy the soliton equation \eqref{sol1}. 
\begin{enumerate}
\item There are no compact expanding solitons of \eqref{gfloweqn}.
\item The only compact steady solitons of \eqref{gfloweqn} are given by torsion-free Spin(7)-structures.
\end{enumerate}
\end{proposition}
\begin{proof}
Taking the trace of \eqref{gsoliton1} gives
\begin{align*}
-R+4\Div T_8+8T_{b;il}T_{i;lb}+8|T_8|^2+2|T|^2-8|T|^2=2\lambda + \Div Y
\end{align*}
which on using the expression for the scalar curvature \eqref{scalar1} simplifies to
\begin{align}\label{gsol3}
-4\Div T_8-6|T|^2=2\lambda + \Div Y.
\end{align}
Integrating \eqref{gsol3} on compact $M$ gives
\begin{align*}
-3\int_M |T|^2\vol = \lambda \text{Vol}(M).
\end{align*}
So $\lambda \leq 0$ and $\lambda=0$ if and only if $T\equiv 0$.
\end{proof}

\printbibliography

@article{dle-isometric,
	adsurl = {https://ui.adsabs.harvard.edu/abs/2021arXiv210906340D},
	archiveprefix = {arXiv},
	author = {{Dwivedi}, Shubham and {Loubeau}, Eric and {Earp}, Henrique N. S{\'a}},
	eprint = {2109.06340},
	journal = {Ann. Sc. Norm. Super. Pisa Cl. Sci. (5) Vol. XXV (2024), 151-215},
	primaryclass = {math.DG},
	title = {{Harmonic flow of $\mathrm{Spin}(7)$-structures}},
	year = 2024}

@misc{Krasnov,
	archiveprefix = {arXiv},
	author = {Kirill Krasnov},
	eprint = {2403.16661},
	primaryclass = {math.DG},
	title = {Dynamics of Cayley Forms},
	year = {2024}}

@article{DeTurck,
	author = {DeTurck, Dennis M.},
	doi = {10.4310/jdg/1214509286},
	fjournal = {Journal of Differential Geometry},
	issn = {0022-040X},
	journal = {J. Differ. Geom.},
	language = {English},
	pages = {157--162},
	title = {Deforming metrics in the direction of their {Ricci} tensors},
	volume = {18},
	year = {1983},
	zbl = {0517.53044},
	zbmath = {3818604}}

@article{FLMS,
	author = {{Fadel}, Daniel and {Loubeau}, Eric and {Moreno}, Andr{\'e}s J. and {Earp}, Henrique N. S{\'a}},
	doi = {https://doi.org/10.1515/crelle-2024-0067},
	journal = {Journal f{\"u}r die reine und angewandte Mathematik (Crelles Journal)},
	number = {817},
	pages = {67-152},
	title = {{Flows of geometric structures}},
	volume = {2024},
	year = 2022}

@book{Chow-Knopf,
	author = {Chow, Bennett and Knopf, Dan},
	fseries = {Mathematical Surveys and Monographs},
	isbn = {0-8218-3515-7},
	issn = {0076-5376},
	language = {English},
	publisher = {Providence, RI: American Mathematical Society (AMS)},
	series = {Math. Surv. Monogr.},
	title = {The {Ricci} flow: an introduction},
	volume = {110},
	year = {2004},
	zbl = {1086.53085},
	zbmath = {2121403}}

@book{AH,
	author = {Andrews, Ben and Hopper, Christopher},
	doi = {10.1007/978-3-642-16286-2},
	fseries = {Lecture Notes in Mathematics},
	isbn = {978-3-642-16285-5; 978-3-642-16286-2},
	issn = {0075-8434},
	language = {English},
	publisher = {Berlin: Springer},
	series = {Lect. Notes Math.},
	title = {The {Ricci} flow in {Riemannian} geometry. {A} complete proof of the differentiable 1/4-pinching sphere theorem},
	volume = {2011},
	year = {2011},
	zbl = {1214.53002},
	zbmath = {5821328}}

@book{Topping,
	author = {Topping, Peter},
	fseries = {London Mathematical Society Lecture Note Series},
	isbn = {0-521-68947-3},
	issn = {0076-0552},
	language = {English},
	publisher = {Cambridge: Cambridge University Press},
	series = {Lond. Math. Soc. Lect. Note Ser.},
	title = {Lectures on the {Ricci} flow},
	volume = {325},
	year = {2006},
	zbl = {1105.58013},
	zbmath = {5062652}}

@article{dgk-flows,
	adsurl = {https://ui.adsabs.harvard.edu/abs/2023arXiv231105516D},
	archiveprefix = {arXiv},
	author = {{Dwivedi}, Shubham and {Gianniotis}, Panagiotis and {Karigiannis}, Spiro},
	doi = {10.48550/arXiv.2311.05516},
	eprint = {2311.05516},
	month = nov,
	primaryclass = {math.DG},
	title = {{Flows of $G_2$-structures, II: Curvature, torsion, symbols, and functionals}},
	year = 2023}

@book{lawson-michelsohn,
	author = {Lawson, Jr., H. Blaine and Michelsohn, Marie-Louise},
	isbn = {0-691-08542-0},
	mrclass = {53-02 (53A50 53C20 57R75 58G10)},
	mrnumber = {1031992},
	mrreviewer = {N. J. Hitchin},
	pages = {xii+427},
	publisher = {Princeton University Press, Princeton, NJ},
	series = {Princeton Mathematical Series},
	title = {Spin geometry},
	volume = {38},
	year = {1989}}

@article{AWW,
	author = {Ammann, Bernd and Weiss, Hartmut and Witt, Frederik},
	doi = {10.1007/s00208-015-1315-8},
	fjournal = {Mathematische Annalen},
	issn = {0025-5831},
	journal = {Math. Ann.},
	language = {English},
	number = {3-4},
	pages = {1559--1602},
	title = {A spinorial energy functional: critical points and gradient flow},
	volume = {365},
	year = {2016},
	zbl = {1358.53052},
	zbmath = {6618541}}

@article{loubeau-saearp,
	author = {Loubeau, Eric and S{\'a} Earp, Henrique N.},
	doi = {10.1007/s10455-023-09928-7},
	fjournal = {Annals of Global Analysis and Geometry},
	issn = {0232-704X},
	journal = {Ann. Global Anal. Geom.},
	language = {English},
	note = {Id/No 23},
	number = {4},
	pages = {42},
	title = {Harmonic flow of geometric structures},
	volume = {64},
	year = {2023},
	zbmath = {7755906}}

@incollection{karigiannis-spin7,
	author = {Karigiannis, Spiro},
	booktitle = {Differential geometry and its applications},
	doi = {10.1142/9789812790613_0023},
	mrclass = {53C44 (53C10 53C29)},
	mrnumber = {2462799},
	mrreviewer = {Francisco Mart\'{\i}n Cabrera},
	pages = {263--277},
	publisher = {World Sci. Publ., Hackensack, NJ},
	title = {Flows of {S}pin(7)-structures},
	url = {https://doi-org.proxy.lib.uwaterloo.ca/10.1142/9789812790613_0023},
	year = {2008}}

@book{joycebook,
	author = {Joyce, Dominic D.},
	isbn = {0-19-850601-5},
	mrclass = {53C29 (14J32 53-01 53-02 53C26 58J60 81T30)},
	mrnumber = {1787733},
	mrreviewer = {Andrew Swann},
	pages = {xii+436},
	publisher = {Oxford University Press, Oxford},
	series = {Oxford Mathematical Monographs},
	title = {Compact manifolds with special holonomy},
	year = {2000}}

@article{dgk-isometric,
	author = {Dwivedi, Shubham and Gianniotis, Panagiotis and Karigiannis, Spiro},
	doi = {10.1007/s12220-019-00327-8},
	fjournal = {Journal of Geometric Analysis},
	issn = {1050-6926},
	journal = {J. Geom. Anal.},
	mrclass = {53C29 (53C25 53E99 58J35 58J60)},
	mrnumber = {4215279},
	number = {2},
	pages = {1855--1933},
	title = {A gradient flow of isometric {$\mathrm{G}_2$}-structures},
	url = {https://doi.org/10.1007/s12220-019-00327-8},
	volume = {31},
	year = {2021}}

@article{fernandez-spin7,
	author = {Fern{\'{a}}ndez, Marisa},
	doi = {10.1007/BF01769211},
	fjournal = {Annali di Matematica Pura ed Applicata. Serie Quarta},
	issn = {0003-4622},
	journal = {Ann. Mat. Pura Appl. (4)},
	mrclass = {53C25 (53C10)},
	mrnumber = {859598},
	mrreviewer = {A. Gray},
	pages = {101--122},
	title = {A classification of {R}iemannian manifolds with structure group {${\rm Spin}(7)$}},
	url = {https://doi-org.proxy.lib.uwaterloo.ca/10.1007/BF01769211},
	volume = {143},
	year = {1986}}

\noindent
Institut f\"ur Mathematik, Humboldt-Universit\"at zu Berlin, Rudower Chaussee 25, 12489 Berlin.\\
\emph{Current address:} Fachbereich Mathematik, Universität Hamburg, Bundesstraße 55, 20146 Hamburg, Germany.\\
\href{mailto:shubham.dwivedi@uni-hamburg.de}{shubham.dwivedi@uni-hamburg.de}

\end{document}